\documentclass[a4paper]{article}
\usepackage{mathptmx}
\usepackage{amscd,amssymb,amsmath,amsthm}

\makeatletter

\renewcommand{\@seccntformat}[1]
{{\csname the#1\endcsname}.\hspace{0.3em}}

\renewcommand{\section}{\@startsection
{section}
{1}
{0mm}
{-1.5\baselineskip}
{\baselineskip}
{\bfseries\normalsize}}

\renewcommand{\subsection}{\@startsection
{subsection}
{2}
{0mm}
{-\baselineskip}
{0.5\baselineskip}
{\normalsize\itshape}}

\renewcommand{\subsubsection}{\@startsection
{subsubsection}
{3}
{0mm}
{-.5\baselineskip}
{-2mm}
{\normalsize\itshape}}

\makeatother

\theoremstyle{plain}
\newtheorem*{theorem*}{Theorem}
\newtheorem{theorem}{Theorem}[section]
\newtheorem*{TAk}{Theorem~$\mathbf{A_k}$}
\newtheorem*{TA1}{Theorem~$\mathbf{A_1}$}
\newtheorem*{TB1}{Theorem~$\mathbf{B_1}$}
\newtheorem*{TCk}{Theorem~$\mathbf{C_k}$}
\newtheorem*{TD1}{Theorem~$\mathbf{D_1}$}

\newtheorem*{TEk}{Theorem~$\mathbf{E_k}$}
\newtheorem*{TE1}{Theorem~$\mathbf{E_1}$}

\newtheorem*{RT}{Regularity Theorem}
\newtheorem*{MT1}{First Maz'ja theorem}
\newtheorem*{MT2}{Second Maz'ja theorem}
\newtheorem*{HL}{Hersch Lemma}
\newtheorem*{GY}{Grigor'yan-Yau theorem}
\newtheorem{lemma}[theorem]{Lemma}
\newtheorem{corollary}[theorem]{Corollary}
\newtheorem{prop}[theorem]{Proposition}
\newtheorem{claim}[theorem]{Claim}
\newtheorem*{corollary*}{Corollary}

\theoremstyle{definition}
\newtheorem*{defin*}{Definition}
\newtheorem{defin}{Definition}[section]

\theoremstyle{remark}
\newtheorem*{remark*}{Remark}
\newtheorem{example}{Example}[section]
\newtheorem*{quest*}{Question}

\DeclareMathAlphabet{\matheur}{U}{eur}{m}{n}
\DeclareMathAlphabet{\matheus}{U}{eus}{m}{n}
\DeclareMathAlphabet{\matheuf}{U}{euf}{m}{n}

\numberwithin{equation}{section}
\newcommand{\abs}[1]{\left\lvert#1\right\rvert}

\DeclareMathOperator{\PSL}{PSL}
\DeclareMathOperator{\CAP}{Cap}
\DeclareMathOperator{\capty}{cap}

\DeclareMathOperator{\INT}{Int}

\begin{document}

\author{Gerasim  Kokarev
\\ {\small\it Mathematisches Institut der Universit\"at M\"unchen }
\\ {\small\it Theresienstr. 39, D-80333 M\"unchen, Germany}
\\ {\small\it Email: {\tt Gerasim.Kokarev@mathematik.uni-muenchen.de}}
}

\title{Variational aspects of Laplace eigenvalues on Riemannian
  surfaces}
\date{}
\maketitle

\begin{abstract}

\noindent
We study the existence and properties of metrics maximising the first Laplace eigenvalue among conformal metrics of unit volume on Riemannian surfaces. We describe a general approach to this problem and its higher eigenvalue versions via the direct method of calculus of variations. The principal results include the general regularity properties of $\lambda_k$-extremal metrics and the existence of a partially regular $\lambda_1$-maximiser.
\end{abstract}

\medskip
\noindent
{\small

 }

\tableofcontents

\section{Introduction}
\subsection{Preliminaries}
Let $M$ be a compact surface, possibly with boundary. For a Riemannian metric $g$ on $M$ we denote by
$$
0=\lambda_0(g)<\lambda_1(g)\leqslant\lambda_2(g)\leqslant\ldots
\leqslant\lambda_k(g)\leqslant\ldots
$$
the eigenvalues of the Laplace operator $-\Delta_g$. When $M$ has a non-empty boundary we assume that the Neumann boundary conditions are imposed. By the result of Korevaar~\cite{Korv}, each eigenvalue $\lambda_k(g)$ is bounded as the metric $g$ ranges in a fixed conformal class on $M$. More precisely, if $M$ is an orientable surface of genus $\gamma$, then there exists an absolute constant $C_*>0$ such that for any Riemannian metric $g$ the following estimate holds
$$
\lambda_k(g)\mathit{Vol}_g(M)\leqslant C_*\cdot k(\gamma+1)
$$
for each $k\geqslant 0$. This is a generalisation of an earlier result by Yang-Yau~\cite{YY80} for the first eigenvalue: for any Riemannian metric $g$
\begin{equation}
\label{YangYau}
\lambda_1(g)\mathit{Vol}_g(M)\leqslant 8\pi\cdot (\gamma+1).
\end{equation}
For genus zero surfaces the result of Hersch~\cite{H70} states that the equality in the inequality above is achieved on the standard round sphere. In~\cite{Be73} Berger asked whether the flat equilateral
torus maximises the quantity  $\lambda_1(g)\mathit{Vol}_g(M)$ among all metrics on the torus. Later Nadirashvili~\cite{Na96} developed an approach to the Berger problem by maximising the first eigenvalues in conformal classes. Since his paper there has been a growing interest in the extremal problems for eigenvalues on surfaces, and in particular, extremal problems in conformal classes. For the progress on the subject we refer to the papers~\cite{CEl,El00,El03} as well as~\cite{JLNNP,JNP06,El06} and references there. 

The previous work~\cite{JLNNP,Na10a} together with numerical evidence indicate  that metrics maximising Laplace eigenvalues are expected to be singular. This poses the following natural questions.

\medskip
\noindent
{\em What singularities of maximal metrics can occur, in principle? Is it possible to describe them?} 

\medskip
\noindent
From the perspective of calculus of variations, the occurrence of singularities means that the class of smooth Riemannian metrics is {\em not natural} for such extremal problems. In other words, there should be developed a new formalism allowing to deal with singular objects. This point of view leads to the questions of the following kind.

\medskip
\noindent
{\em What is an appropriate variational setting for the eigenvalue extremal problems on singular metrics? In particular, what is the right notion of extremality for singular metrics?}

\medskip
\noindent
One of the purposes of this paper is to develop a general setting to address a circle of similar problems. Below we describe its content  in more detail.

\subsection{Outline of the results}
We study the existence and properties of metrics maximising the first eigenvalue $\lambda_1(g)\mathit{Vol}_g(M)$, and more generally, the $k$th eigenvalue $\lambda_k(g)\mathit{Vol}_g(M)$, among conformal metrics on Riemannian surfaces. More precisely, the purpose of this paper is to develop an approach to this problem via  the {\it direct method of calculus of variations}. First, we show that the Laplace eigenvalues $\lambda_k(g)$ naturally extend to  `weak conformal metrics', understood as Radon measures and prove bounds for them (Theorems~${A_k}$ and~${A_1}$). This setting of {\em eigenvalue problems on surfaces with measures} gives a uniform formalism of treating eigenvalue problems on singular surfaces as well as eigenvalue problems with Steklov boundary conditions. We also prove a general existence theorem (Theorem $B_1$) of a measure maximising the first non-trivial eigenvalue $\lambda_1$ under the hypothesis
\begin{equation}
\label{intro:ru}
\sup\left\{\lambda_1(g)\mathit{Vol}_g(M):g\in c\right\}>8\pi.
\end{equation}
on a given conformal class $c$. The hypothesis~\eqref{intro:ru} guarantees that the maximiser is not pathologically singular. It satisfies a {\em linear isocapacitory inequality}, see Sect.~\ref{nonvan}; in particular, it vanishes on sets of zero capacity and the mass of balls $\mu(B(x,r))$ decays at least as $\ln^{-1}(1/r)$ as $r\to 0+$. 

Second, we define a notion of $\lambda_k$-extremality of general measures under "conformal variations" and derive first variation formulas. The main result of the paper is concerned with the study of regularity properties of {\em general $\lambda_k$-extremal measures}. More precisely, in Sect.~\ref{em1} we prove the following statement (Theorem~$C_k$).
\begin{RT}
Let $M$ be a compact surface, possibly with boundary, endowed with a conformal class $c$ of Riemannian metrics. Let $\mu$ be a $\lambda_k$-extremal measure  which is not completely singular and such that the embedding
\begin{equation}
\label{intro:em}
L_2(M,\mu)\cap L^1_2(M,\mathit{Vol})\subset L_2(M,\mu)
\end{equation}
is compact. 
\begin{itemize}
\item [(i)] Then the measure $\mu$ is absolutely continuous (with respect to $\mathit{Vol}_g$, $g\in c$) in the interior of its support $S\subset M$, its density function is $C^\infty$-smooth in $S$ and vanishes at isolated points only. In other words, the measure $\mu$ defines a $C^\infty$-smooth metric on $S$, conformal to $g\in c$ away from isolated degeneracies which are conical singularities.
\item [(ii)] If the support of the measure $\mu$ does not coincide with $M$, then the measure has a non-trivial singular set $\Sigma\subset M\backslash\INT S$.
\end{itemize}
\end{RT}

It is important to mention that there are singular $\lambda_1$-extremal measures, see Sect.~\ref{em1}, and thus, the regularity theory is non-trivial. The compactness of embedding~\eqref{intro:em} in the theorem is a delicate hypothesis. It is closely related to the behaviour of sharp constants in the so-called isocapacitory inequalities. Studying this relationship, we obtain asymptotics for the values $\mu(B(x,r))$ as $r\to 0$, which describe the margin between the validity and the failure of the compactness of embedding~\eqref{intro:em}. These asymptotics show that there are capacitory measures for which embedding~\eqref{intro:em} is not compact. 

As an elementary application of the developed analysis, we obtain the notion of $\lambda_k$-extremality for metrics with conical singularities under conformal deformations, and are able to characterise such metrics via harmonic maps into unit spheres in the Euclidean space, see Corollary~\ref{hm}. The latter statement generalises earlier results in~\cite{El03}, see also~\cite{El00,Na96}, known for Riemannian metrics. 

In the final part of the paper, we prove the existence of a partially regular $\lambda_1$-maximiser (Theorem~$D_1$) and study concentration-compactness properties of $\lambda_k$-extremal metrics. The version of the latter result for the first eigenvalue (Theorem~$E_1$) says that any sequence $g_n$ of $\lambda_1$-extremal conformal metrics contains a subsequence that either converges smoothly to a $\lambda_1$-extremal metric or concentrates to a pure Dirac measure and $\lambda_1(g_n)\to 8\pi$ as $n\to +\infty$. 

After the first preprint of the paper has appeared, there has been a number of developments on the subject. First, the results in our Example~\ref{stek} (the Steklov eigenvalue problem) have been independently obtained in~\cite{CEG}. Extremal problems for Steklov eigenvalues have been also studied by Fraser and Schoen in~\cite{FS2} where the authors prove the existence of a $\lambda_1$-maximiser for zero genus surfaces. The state of the subject concerning extremal problems for Laplace eigenvalues is also described in~\cite{FS3}. In the recent preprint~\cite{Pet} Petrides claims a general existence theorem of a $\lambda_1$-maximiser in every conformal class on a closed Riemannian surface, the statement also announced by Nadirashili and Sire~\cite{Na10a}. The argument by Petrides uses the non-concentration analysis from the present paper as well as the heat kernel regularization introduced by Fraser and Schoen~\cite{FS2}. Petrides also shows that hypothesis~\eqref{intro:ru} in our Theorem~$B_1$ always holds on closed Riemannian surfaces different from a sphere. On the other hand, by~\cite{Jam,KoNa} on surfaces with boundary there are conformal classes for which this hypothesis fails. 

In spite of all this progress made within the last 2-3 years, we have kept the main text of the paper essentially unchanged making only the corrections requested by the referee.

\subsection{Paper organisation}
The paper is organised in the following way. In Sect.~\ref{prems} we describe a general setup for the variational problem. First, we show that Laplace eigenvalues naturally extend to the set of Radon measures (which play the role of "weakly conformal metrics") where they are upper semi-continuous in the weak topology. We also discuss the boundedness of eigenvalues among non-atomic probability measures, based on earlier results by Korevaar and their improvements by Grigor'yan, Netrusov, and Yau.

In Sect.~\ref{nonvan} we study properties of the measures whose first eigenvalues do not vanish. We show that this hypothesis is equivalent to the validity of a linear isocapacitory inequality (Corollary~\ref{lisocap}). We proceed with comparing it with the compactness hypothesis for embedding~\eqref{intro:em}; our methods here are based on the isocapacitory inequalities and the results by Maz'ja. In Sect.~\ref{weakmax:ex} we give a general statement on the existence of a $\lambda_1$-maximal Radon measure. Sect.~\ref{em1} is devoted to the actual calculus of variations -- we define a notion of extremality and derive the first variation formulas (Lemma~\ref{d:eiv}) for an arbitrary eigenvalue $\lambda_k$. These are then used to prove the regularity of any $\lambda_k$-extremal metric under the hypothesis that the embedding~\eqref{intro:em} is compact. In Sect.~\ref{prm} we give an elementary argument which yields the existence of partially regular maximisers in a conformal class.

The principal part of the paper ends with a collection of other related results and remarks in Sect.~\ref{other}. These include the concentration-compactness properties of extremal metrics, geometric hypotheses allowing to obtain better regularity, and a number of open questions. The paper contains two appendices where we collect details of technical or complementary nature for reader's convenience.

\medskip
\noindent
{\it Acknowledgements.}  During the course of the work I have benefited from the comments and advice of Vladimir Eiderman, Alexander Grigor'yan, Emmanuel Hebey, Nikolai Nadirashvili,  and Iosif Polterovich. The work has been accomplished during author's stay at the University of Cergy-Pontoise (France) during 2010/11 supported by the EU Commission via the Marie Curie Actions scheme.

\section{Eigenvalues on measure spaces}
\label{prems}
\subsection{Classical notation}
Let $M$ be a compact smooth surface with or without boundary. Recall that for a Riemannian metric $g$ on $M$ the Laplace operator $-\Delta_g$ in local coordinates $(x^i)$, $1\leqslant i\leqslant 2$, has the form
$$
-\Delta_g=-\frac{1}{\sqrt{\abs{g}}}\frac{\partial~}{\partial x^i}\left(\sqrt{\abs{g}}g^{ij}\frac{\partial~}{\partial x^j}\right),
$$
where $(g_{ij})$ are components of the metric $g$, $(g^{ij})$ is the inverse tensor, and $\abs{g}$ stands for $\det(g_{ij})$. Above we use the summation convention for the repeated indices. The Laplace eigenvalues 
$$
0=\lambda_0(g)<\lambda_1(g)\leqslant\ldots\leqslant\lambda_k(g)\leqslant\ldots
$$
are real numbers for which the equation
\begin{equation}
\label{eifun}
(\Delta_g+\lambda_k(g))u=0
\end{equation}
has a non-trivial solution. In the case when $M$ has a non-empty boundary, we suppose that the solutions $u$ above satisfy Neumann boundary conditions. The solutions of equation~\eqref{eifun} are called eigenfunctions, and their collection over all eigenvalues forms a complete orthogonal basis in $L^2(M)$. Recall that by variational characterisation 
\begin{equation}
\label{minmax}
\lambda_k(g)=\inf_{\Lambda^{k+1}}\sup_{u\in\Lambda^{k+1}}\matheur R_g(u),
\end{equation}
where the infimum is taken over all $(k+1)$-dimensional subspaces in $C^{\infty}(M)$, the supremum is over non-trivial $u\in\Lambda^{k+1}$, and $\matheur R_g(u)$ stands for the Rayleigh quotient,
$$
\matheur R_g(u)=\left(\int_M\abs{\nabla u}^2\mathit{dVol}_g\right)/\left(\int_M u^2\mathit{dVol}_g\right).
$$
The infimum in relation~\eqref{minmax} is achieved on the space spanned by the first $(k+1)$ eigenfunctions.

\subsection{The setup for measure spaces. Korevaar eigenvalue bounds.}
Let $M$ be a compact surface and $c$ be a conformal class of $C^\infty$-smooth metrics on $M$. The conformal metrics from $c$ can be identified with their volume measures, and to apply variational methods, we
consider eigenvalues as functionals of more general measures on $M$. The reasoning is that the space of conformal Riemannian metrics does not possess any compactness properties and, in fact,  is not even closed in any natural topology. Besides, we expect that maximal metrics (that is eigenvalues maximisers) may be degenerate, see~\cite{JLNNP,Na10a},  and we should be able to assign the values $\lambda_k$ to such metrics.

For a Radon measure $\mu$ on $M$ the $k$th eigenvalue $\lambda_k(\mu,c)$ is defined by the min-max principle 
$$
\lambda_k(\mu,c)=\inf_{\Lambda^{k+1}}\sup_{u\in\Lambda^{k+1}}\matheur R_c(u,\mu),
$$
where the infimum is taken over all $(k+1)$-dimensional subspaces $\Lambda^{k+1}\subset L_2(M,\mu)$ formed by $C^{\infty}$-smooth functions, the supremum is over non-trivial $u\in\Lambda^{k+1}$, and $\matheur R_c(u,\mu)$ stands for the Rayleigh quotient
\begin{equation}
\label{RQ}
\matheur R_c(u,\mu)=\left(\int_M\abs{\nabla
  u}^2\mathit{dVol}_g\right)/\left(\int_Mu^2d\mu\right),
\end{equation}
where $g\in c$ is a reference metric. If $M$ has a non-empty boundary, we assume that the test functions are continuous up to the boundary. By conformal invariance of the Dirichlet energy, the Rayleigh quotient does not depend on a choice of such a metric $g\in c$.

The following example shows that so defined eigenvalues are natural generalisations of Laplace eigenvalues to certain degenerate metrics.
\begin{example}[Metrics with conical singularities]
Let $M$ be a compact surface, possibly with boundary, and $h$ be a metric on $M$ with conical singularities. Then, as is known, such a metric $h$ is conformal to a genuine Riemannian metric $g$ on $M$ away from the singularities. The Dirichlet integral with respect to the metric $h$ is defined as an improper integral; by the conformal invariance, it satisfies the relation
$$
\int_M\abs{\nabla u}^2\mathit{dVol}_h=\int_M\abs{\nabla u}^2\mathit{dVol}_g
$$
for any smooth function $u$. Thus, we conclude that the Laplace eigenvalues of a metric $h$ coincide with the eigenvalues of the pair $(\mathit{Vol}_h,[g])$ in the sense introduced above. Mention also that the $\lambda_k(\mathit{Vol}_h,[g])$'s coincide with other definitions of Laplace eigenvalues for metrics with conical singularities used in the literature, see e.g.~\cite{JLNNP, Ko3}.  
\end{example}

Clearly, the zero eigenvalue $\lambda_0(\mu,c)$ vanishes for any measure $\mu$ and any conformal class $c$. The corresponding eigenfunctions coincide with constant functions. The following example shows that for higher eigenvalues the eigenfunctions (orthogonal to constants) do not always exist.
\begin{example}[Possible pathologies]
Let $\mu$ be a discrete measure supported at $\ell$ distinct points. Since the capacity of each point is equal to zero, it is straightforward to show that
$$
\lambda_k(\mu,c)=\left\{
\begin{array}{cc}
0, & \text{if }\ell>k,\\
+\infty, & \text{if }\ell\leqslant k,
\end{array}
\right.
$$
for an arbitrary conformal class $c$ on $M$.
\end{example}
Despite this example, it is straightforward to see that the $k$th eigenvalue $\lambda_k(\mu,c)$ is finite for any measure whose support contains more than $k$ distinct points. Further, the following result shows that the quantity $\lambda_k(\mu,c)\mu(M)$ is actually uniformly bounded for all continuous (that is with trivial discrete part) Radon measures $\mu$.
\begin{TAk}
Let $M$ be a compact surface, possibly with boundary, endowed with a conformal class $c$. Then there exists a constant $C>0$ such that for any continuous Radon measure $\mu$ the following inequality holds:
$$
\lambda_k(\mu,c)\mu(M)\leqslant C k.
$$
Moreover, if $M$ is orientable, then the constant $C$ can be chosen independently on the conformal class $c$ in the form $C_*(\gamma+1)$, where $C_*>0$ is a universal constant, and $\gamma$ is the genus of
$M$. 
\end{TAk}
The theorem above is a basis for our variational approach. Its proof is based on the results by Grigor'yan, Netrusov, and Yau~\cite{GY99,GNY}, built on the original method of Korevaar~\cite{Korv}. It appears in Appendix~\ref{ap:gy}. The estimate~\eqref{YangYau} of Yang and Yau can be also generalised for continuous Radon measures to give a more precise version of Theorem~$A_k$ for the first eigenvalue, see~\cite{KoNa}.
\begin{TA1}
Let $M$ be an orientable compact surface, possibly with boundary, endowed with a conformal class $c$. Then for any continuous Radon measure $\mu$ the first eigenvalue satisfies the inequality
$$
\lambda_1(\mu,c)\mu(M)\leqslant 8\pi(\gamma+1),
$$
where $\gamma$ is the genus of $M$.
\end{TA1} 
\begin{example}[Steklov eigenvalues]
\label{stek}
Let $M$ be a surface with boundary, endowed with a conformal class $c$. For a Riemannian metric $g\in c$ let $\mu_g$ be its boundary volume measure. Then the eigenvalues $\lambda_k(\mu_g,c)$ coincide with the so-called Steklov eigenvalues of a metric $g$, representing the spectrum of the Dirichlet-to-Neumann map. We refer to the recent papers~\cite{GP,FS} for the account and further references on the subject. In particular, Theorems~$A_k$ and~$A_1$ above yield isoperimetric inequalities for the Steklov eigenvalues, complementing earlier results by Weinstock~\cite{Wei} and Fraser and Schoen~\cite{FS}.
\end{example}

Now the existence problem for a maximising $\lambda_k(g)\mathit{Vol}_g(M)$ metric in $c$ splits into the two  separate parts: the existence of a weak maximiser -- that is a continuous Radon measure maximising the quantity $\lambda_k(\mu,c)\mu(M)$ among all continuous Radon measures, and the regularity theory for weak
maximisers. The following {\em upper semi-continuity} property is an important ingredient for the former.
\begin{prop}[Upper semi-continuity]
\label{up_semi}
Let $(M,c)$ be a compact Riemann surface, and $(\mu_n)$, $n=1,2,\ldots$, be a sequence of Radon probability measures on $M$ converging weakly to a Radon probability measure $\mu$. Then for any $k\geqslant 0$ we have
$$
\lim\sup\lambda_k(\mu_n,c)\leqslant\lambda_k(\mu,c).
$$
\end{prop}
\begin{proof}
For a given $\varepsilon>0$, let $\Lambda^{k+1}$ be a $(k+1)$-dimensional subspace of $C^\infty(M)$ such that
$$
\sup_{u\in\Lambda^{k+1}}\matheur R_c(\mu,u)\leqslant\lambda_k(\mu,c)+\varepsilon.
$$
By weak convergence of measures, we obtain that
$$
\sup_{u\in\Lambda^{k+1}}\matheur R_c(\mu_n,u)\longrightarrow\sup_{u\in\Lambda^{k+1}}\matheur R_c(\mu,u).
$$
In other words, for a sufficiently large $n$ we have
$$
\sup_{u\in\Lambda^{k+1}}\matheur R_c(\mu_n,u)\leqslant\sup_{u\in\Lambda^{k+1}}\matheur R_c(\mu,u)+\varepsilon\leqslant\lambda_k(\mu,c)+2\varepsilon.
$$
The latter implies that
$$
\lambda_k(\mu_n,c)\leqslant\lambda_k(\mu,c)+2\varepsilon
$$
for all sufficiently large $n$, and passing to the limit, we obtain
$$
\lim\sup\lambda_k(\mu_n,c)\leqslant\lambda_k(\mu,c)+2\varepsilon.
$$
Since $\varepsilon>0$ above is arbitrary, we are done.
\end{proof}

\subsection{Preliminaries on eigenfunctions}
Here we  collect a number of elementary statements describing properties of eigenfunctions in the setting of measure spaces. We start with introducing a natural space for the Rayleigh quotient~\eqref{RQ}, that is  the space
$$
\mathcal L=L_2(M,\mu)\cap L_2^1(M,\mathit{Vol}_g);
$$
here the second space in the intersection is formed by distributions whose first derivatives are in $L_2(M,\mathit{Vol}_g)$, see~\cite{Ma}. Following classical terminology, a function $u\in\mathcal L$ is called an {\it eigenfunction} for $\lambda_k(\mu,c)$, if it is contained in a $(k+1)$-dimensional subspace $\Lambda^{k+1}\subset\mathcal L$ such that
\begin{equation}
\label{def:eigen}
\matheur R_c(u,\mu)=\sup_{\varphi\in\Lambda^{k+1}}\matheur R_c(\varphi,\mu)
\end{equation}
and the value $\matheur R_c(u,\mu)$ coincides with $\lambda_k(\mu,c)$. The following characterisation of eigenfunctions is often used in the sequel.
\begin{prop}
\label{prop:eigen}
Let $M$ be a compact surface, possibly with boundary, endowed with a conformal class of Riemannian metrics. Let $\mu$ be a continuous Radon measure on $M$ whose eigenvalue $\lambda_k(\mu,c)$ is positive.
Suppose that there exist eigenfunctions corresponding to the first $k$ eigenvalues $\lambda_\ell(\mu)$, $0<\ell<k$. Then a non-trivial function $u\in\mathcal L$ is an eigenfunction for $\lambda_k(\mu,c)$ if and only if it satisfies the integral identity 
\begin{equation}
\label{eq:eigen}
\int_M\langle\nabla u,\nabla\varphi\rangle\mathit{dVol}_g
=\lambda_k(\mu,c)\int_Mu\cdot\varphi d\mu
\end{equation}
for any test-function $\varphi\in\mathcal L$.
\end{prop}
\begin{proof}
Let $u$ be an eigenfunction for $\lambda_k(\mu,c)$, and denote by $\Lambda^{k+1}$ the span of eigenfunctions corresponding to $\lambda_\ell(\mu,c)$, where $0\leqslant\ell\leqslant k$. For a
test-function $\varphi\in\Lambda^{k+1}$ the function 
\begin{equation}
\label{R}
t\longmapsto\matheur R_c(u+t\varphi,\mu)
\end{equation}
has a maximum at $t=0$, and relation~\eqref{eq:eigen} follows by differentiation of the Rayleigh quotient at $t=0$. Further for a test-function $\varphi$ from the orthogonal complement of $\Lambda^{k+1}$
in $\mathcal L$ the function~\eqref{R} has a minimum at $t=0$, and the conclusion follows in the same fashion.

Conversely, suppose that a function $u$ satisfies identity~\eqref{eq:eigen} for any $\varphi\in\mathcal L$. Then, in particular, the value of the Rayleigh quotient $R_c(u,\mu)$ coincides with $\lambda_k(\mu,c)$. The $(k+1)$-dimensional space containing $u$ and satisfying~\eqref{def:eigen} can be constructed as a span of $u$ with eigenfunctions corresponding to lower eigenvalues as well as eigenvalues that coincide with $\lambda_k(\mu,c)$.
\end{proof}
Note that the hypothesis on the existence of lower eigenfunctions, in Prop.~\ref{prop:eigen}, is vacuous for the first eigenvalue.  In general, the existence of eigenfunctions is related to the compactness of the embedding
\begin{equation}
\label{CE}
\mathcal{L}=L_2(M,\mu)\cap L_2^1(M,\mathit{Vol}_g)\subset L_2(M,\mu).
\end{equation}
The following statement follows by fairly standard arguments; we outline them for the sake of completeness.
\begin{prop}
\label{ex:eigen}
Let $M$ be a compact surface, possibly with boundary, endowed with a conformal class of Riemannian metrics, and $\mu$ be a Radon measure such that the embedding~\eqref{CE} is compact. Then for any $k>0$ the
eigenvalue $\lambda_k(\mu,c)$ is positive and has an eigenfunction.  Moreover, the space formed by eigenfunctions corresponding to equal eigenvalues is finite-dimensional. 
\end{prop}
\begin{proof}
We prove the theorem by induction in $k$. The statement on the existence of eigenfunctions is, clearly, true for $k=0$. Suppose the eigenfunctions exist for any $\ell\leqslant (k-1)$; there is a collection of pair-wise orthogonal eigenfunctions $\varphi_\ell$ corresponding to $\lambda_\ell(\mu)$, where $\ell\leqslant (k-1)$. We are to prove the existence of an eigenfunction for $\lambda_k(\mu)$ which is orthogonal to the span of the $\varphi_\ell$'s.

Let $(u_n)$ be a minimising sequence for the Rayleigh quotient $\matheur R_c(u,\mu)$ in the orthogonal complement of the span of the $\varphi_\ell$'s;
$$
\int_Mu_n^2d\mu=1,\quad \int_M\abs{\nabla  u_n}^2\mathit{dVol}_g\longrightarrow\lambda_k(\mu), \quad\text{ as
}n\to +\infty.
$$
Since the embedding~\eqref{CE} is compact, we conclude that $(u_n)$ contains a subsequence
converging weakly in $L^1_2(M,\mathit{Vol}_g)$ and strongly in $L_2(M,\mu)$ to a function $u\in\mathcal L$. Clearly, the limit function $u$ is orthogonal to the span of the $\varphi_\ell$'s, and its norm in $L_2(M,\mu)$ equals one. By lower semi-continuity of the Dirichlet energy, we further obtain
$$
\int_M\abs{\nabla  u}^2\mathit{dVol}_g\leqslant\lim\inf\int_M\abs{\nabla u_n}^2\mathit{dVol}_g=\lambda_k(\mu).
$$
Thus, we conclude that the function $u$ is indeed a minimiser for the Rayleigh quotient $\matheur R_c(u,\mu)$ among functions orthogonal to the span of the $\varphi_\ell$'s.

The statement on the dimension of eigenfunctions corresponding to equal eigenvalues follows by the same compactness argument.
\end{proof}
The existence of eigenfunctions lies at the heart of our method establishing the regularity of extremal metrics in Sect.~\ref{em1}. The hypotheses ensuring the existence are related to the so-called Maz'ja isocapacitory inequalities and studied in more detail in the following section.

\section{Measures with non-vanishing first eigenvalue}
\label{nonvan}
\subsection{No atoms lemma}
In this section we study Radon measures on $M$ with non-vanishing first eigenvalue. To avoid dealing with trivial pathologies we always assume that the measures under consideration are not Dirac measures. The first useful result shows that such measures have to be continuous, that is with trivial discrete part.
\begin{lemma}
\label{comp}
Let $(M,c)$ be a compact Riemann surface, possibly with boundary. Let $\mu$ be a non-continuous Radon measure on $M$ that is not a pure Dirac measure. Then the first eigenvalue $\lambda_1(\mu,c)$ vanishes.
\end{lemma}
\begin{proof}
For the sake of simplicity, we prove the lemma for the case when $M$ is closed only. Let $x\in M$ be a point of positive mass, $m=\mu(x)>0$. Denote by $\mu_*$ the measure $(\mu-m\delta_x)$, and let $\Omega$ be a coordinate ball around $x$ such that $\delta=\mu_*(M\backslash\Omega)$ is strictly positive. Since the capacity of a point is zero, then for a given $\varepsilon>0$ there exists a function $\varphi\in C^\infty_0(\Omega)$ such that $0\leqslant\varphi\leqslant 1$,
$$
\varphi=1\text{ in a neighbourhood of }x,\quad\text{and }\int_M\abs{\nabla\varphi}^2d\mathit{Vol}_g<\varepsilon.
$$
The integral above refers to a fixed metric $g\in c$. Denote by $\alpha$ the mean-value of the function $\varphi$,
$$
\alpha=\int_M\varphi\mathit{dVol}_g>0.
$$
Then by variational principle, we have
$$
\lambda_1(\mu,c)\int_M(\varphi-\alpha)^2d\mu\leqslant\int_M\abs{\nabla\varphi}^2\mathit{dVol}_g.
$$
The right-hand side is not greater than $\varepsilon$, and due to the choice of $\varphi$, we obtain
$$
\lambda_1(\mu,c)(\alpha^2\delta+(1-\alpha)^2m)\leqslant\varepsilon.
$$
By elementary analysis, the left-hand side above is bounded below by the quantity
$$
\left(\lambda_1(\mu,c)m\delta\right)/\left(m+\delta\right)\geqslant 0.
$$
Since $m$ and $\delta$ are strictly positive, and $\varepsilon$ is arbitrary, we conclude that the first eigenvalue $\lambda_1(\mu,c)$ has to vanish.
\end{proof}

\subsection{Bounds via fundamental tone and  isocapacitory inequalities}
We proceed with showing that measures with non-vanishing first eigenvalue satisfy certain Poincare inequalities.  The latter are closely related to the the notion of the fundamental tone, which we recall now.

For a subdomain $\Omega\subset M$ with non-empty boundary the {\em fundamental tone} $\lambda_*(\Omega,\mu)$ is defined as the infimum of the Rayleigh quotient $\matheur R_c(u,\mu)$ over all smooth functions supported in $\Omega$. The following lemma gives bounds for the first eigenvalue in terms of the fundamental tone; a similar statement in a slightly different context can be found in~\cite{Chen}.
\begin{lemma}
\label{tone}
Let $M$ be a compact surface, possibly with boundary, endowed with a conformal class of Riemannian metrics, and $\mu$ be a Radon probability measure on $M$. Then, we have 
$$
\inf\lambda_*(\Omega,\mu)\leqslant\lambda_1(\mu,c)\leqslant 2\inf\lambda_*(\Omega,\mu),
$$
where the infimums are taken over all subdomains $\Omega\subset M$ such that $0<\mu(\Omega)\leqslant 1/2$.
\end{lemma}
\begin{proof}
First we prove the upper bound. Let $u$ be a smooth function supported in $\Omega$, and we suppose that the integral $\int u^2d\mu$ equals one. Denote by $\bar u$ its mean value, that is the integral $\int ud\mu$. Then we have
$$
\int(u-\bar u)^2d\mu=1-\bar u^2\geqslant 1-\left(\int u^2d\mu\right)\cdot\mu(\Omega)
=1-\mu(\Omega)=\mu(M\backslash\Omega).
$$
From this, we conclude that
$$
\lambda_1(\mu,c)\leqslant\lambda_*(\Omega,\mu)/\mu(M\backslash\Omega).
$$
Since the domain $\Omega$ is arbitrary, we further obtain
\begin{multline*}
\lambda_1(\mu,c)\leqslant\inf_{0<\mu(\Omega)<1}\min\left\{\lambda_*(\Omega,\mu)/\mu(M\backslash\Omega), \lambda_*(M\backslash\Omega,\mu)/\mu(\Omega)\right\}\\
\leqslant\inf_{0<\mu(\Omega)\leqslant 1/2}\lambda_*(\Omega,\mu)/\mu(M\backslash\Omega)\leqslant 2\inf_{0<\mu(\Omega)\leqslant 1/2}\lambda_*(\Omega,\mu).
\end{multline*}
We proceed with demonstrating the lower bound. Let $u$ be a test-function for the first eigenvalue, that is
\begin{equation}
\label{test}
\int u^2d\mu=1\quad\text{and}\quad\int ud\mu=0.
\end{equation}
Let $c$ be a median of $u$, that is a real number such that
$$
\mu(u<c)\leqslant 1/2\quad\text{and}\quad\mu(u>c)\leqslant 1/2.
$$
Denote by $u_c^+$ and $u_c^-$ the non-negative and non-positive parts of $(u-c)$, and by $\Omega^\pm$ their supports respectively. First, note that
$$
\int\abs{\nabla u}^2d\mathit{Vol}_g=\int\abs{\nabla u_c^+}^2d\mathit{Vol}_g+
\int\abs{\nabla u_c^-}^2d\mathit{Vol}_g.
$$
Using this relation, we obtain
\begin{multline*}
\matheur R_c(u,\mu)\geqslant\lambda_*(\Omega^+)\int(u_c^+)^2d\mu+\lambda_*(\Omega^-)\int(u_c^-)^2d\mu\\
\geqslant\inf_{0<\mu(\Omega)\leqslant 1/2}\lambda_*(\Omega)\left(\int(u_c^+)^2d\mu+\int(u_c^-)^2d\mu\right)\\
=\inf_{0<\mu(\Omega)\leqslant 1/2}\lambda_*(\Omega)\int(u_c^+-u_c^-)^2d\mu=\inf_{0<\mu(\Omega)\leqslant 1/2}\lambda_*(\Omega)\int(u-c)^2d\mu.
\end{multline*}
By~\eqref{test} the last integral clearly equals $(1+c^2)$, and we conclude that 
$$
\matheur R_c(u,\mu)\geqslant\inf_{0<\mu(\Omega)\leqslant 1/2}\lambda_*(\Omega).
$$
Taking the infimum over all test-functions, we thus get the lower bound for $\lambda_1(u,\mu)$.
\end{proof}
One of the consequences of this lemma is the characterisation of measures with non-vanishing first eigenvalue $\lambda_1(\mu,c)$ via isocapacitory inequalities. To explain this we introduce more notation.

Let $\Omega\subset M$ be an open subdomain. For any compact set $F\subset\Omega$ the capacity $\CAP(F,\Omega)$ is defined as
$$
\CAP(F,\Omega)=\inf\left\{\int\abs{\nabla\varphi}^2\mathit{dVol}_g : \varphi\in C_0^\infty(\Omega), \varphi\equiv 1\text{ on }F \right\}.
$$
Further, by the {\em isocapacity constant} $\beta(\Omega,\mu)$ of $\Omega$ we call the quantity
$$
\sup\left\{\mu(F)/\CAP(F,\Omega): F\subset\Omega\text{ is a compact set }\right\}.
$$
By the results of Maz'ja~\cite[Sect.~2.3.3]{Ma}, see also~\cite{Chen}, the isocapacity constant and the fundamental tone are related by the following inequalities:
\begin{equation}
\label{MaIn}
(4\beta(\Omega,\mu))^{-1}\leqslant\lambda_*(\Omega,\mu)\leqslant(\beta(\Omega,\mu))^{-1}.
\end{equation}
Combining these with Lemma~\ref{tone}, we obtain the following corollary.
\begin{corollary}
\label{iff}
Under the hypotheses of Lemma~\ref{tone}, we have
$$
\inf (4\beta(\Omega,\mu))^{-1}\leqslant\lambda_1(\mu,c)\leqslant 2\inf (\beta(\Omega,\mu))^{-1},
$$
where the infimums are taken over all subdomains $\Omega\subset M$ such that $0<\mu(\Omega)\leqslant 1/2$. In particular, the first eigenvalue $\lambda_1(\mu,c)$ is positive if and only if the isocapacity constant $\beta(\Omega,\mu)$ is bounded as $\Omega$ ranges over all subdomains such that $0<\mu(\Omega)\leqslant 1/2$.
\end{corollary}
As another consequence, we mention the following statement.
\begin{corollary}[Linear isocapacitory inequality]
\label{lisocap}
Under the hypotheses of Lemma~\ref{tone}, the first eigenvalue $\lambda_1(\mu,c)$ does not vanish if and only if there exists a positive constant $C>0$ such that the measure $\mu$ satisfies the following inequality
$$
\mu(F)\leqslant C\cdot\CAP(F,\Omega)
$$ 
for any closed subset $F\subset\Omega$ and any subdomain $\Omega$ such that $0<\mu(\Omega)\leqslant 1/2$. In particular, if a measure $\mu$ with non-vanishing $\lambda_1(\mu,c)$ is not a pure Dirac measure, then it vanishes on sets of zero capacity.
\end{corollary}
The last statement of the corollary follows from the linear isocapacitory inequality together with Lemma~\ref{comp}. The linear isocapacitory inequality also implies that 
$$
\mu(B(x,r))\leqslant C_*\cdot\ln^{-1}(1/r)
$$ 
for some constant $C_*$ and all sufficiently small $r>0$. The last relation can be also obtained directly from the hypothesis $\lambda_1(\mu,c)>0$ by constructioning appropriate test-functions, thus avoiding Lemma~\ref{tone} and the Maz'ja inequality~\eqref{MaIn}.

\subsection{Existence of eigenfunctions and Maz'ja theorems}
As we know, see Sect.~\ref{prems}, the existence of eigenfunctions is ensured by the {\em compact embedding}
of the spaces
\begin{equation}
\label{ce}
\mathcal{L}=L_2(M,\mu)\cap L_2^1(M,\mathit{Vol}_g)\subset L_2(M,\mu).
\end{equation}
In this section we describe necessary and sufficient conditions for this hypothesis. First, recall that a Radon measure is called {\em completely singular} if it is supported in a Borel set $\Sigma$ of zero Lebesgue measure, that is $\mu (M\backslash\Sigma)=0$. The measures that are not completely singular are precisely the measures with non-trivial absolutely continuous parts. The following auxiliary lemma reduces the compactness question to the compact embedding results, obtained by Maz'ja in~\cite{Ma}.
\begin{lemma}
\label{re2ma}
Let $M$ be a compact surface, possibly with boundary, endowed with a conformal class $c$ of Riemannian metrics, and $\mu$ be a Radon measure on $M$.
\begin{itemize}
\item[(i)] Suppose that the embedding~\eqref{ce} is compact. Then the space $W^{1,2}(M,\mathit{Vol}_g)$, where $g\in c$, embeds compactly into $L_2(M,\mu)$.
\item[(ii)] Conversely, suppose that the measure $\mu$ is not completely singular, has a positive first eigenvalue $\lambda_1(\mu,c)$, and the space $W^{1,2}(M,\mathit{Vol}_g)$ embeds compactly into $L_2(M,\mu)$. Then the embedding~\eqref{ce} is compact.
\end{itemize}
\end{lemma}
\begin{proof}
We start with the proof of the statement~$(i)$; it is sufficient to show that any sequence~$(u_n)$ bounded in $W^{1,2}(M,\mathit{Vol}_g)$ is also bounded in the space $L_2(M,\mu)$. Since the embedding~\eqref{ce} is compact, by Prop.~\ref{ex:eigen} the first eigenvalue is positive, and by Lemma~\ref{tone} so is the fundamental tone $\lambda_*(\Omega)$ of any sufficiently small subdomain $\Omega\subset M$. Let $(\Omega_i)$ be a finite covering of $M$ by such subdomains, and $(\varphi_i)$ be the corresponding partition of unity. Then we obtain
\begin{multline*}
\int(u_n\varphi_i)^2d\mu\leqslant 2\lambda_*^{-1}(\Omega_i)\left(\int\abs{\nabla u_n}^2\varphi_i^2 \mathit{dVol}_g +\int\abs{\nabla\varphi_i}^2u_n^2\mathit{dVol}_g\right) \\
\leqslant C_i\left(\int\abs{\nabla u_n}^2\mathit{dVol}_g+\int u_n^2\mathit{dVol}_g\right),
\end{multline*}
where the positive constant $C_i$ depends on $\lambda_*(\Omega_i)$ and the $\varphi_i$, and the claim follows by summing up these inequalities.

Now we demonstarte the statement~$(ii)$. First, denote by $\mathcal L_0$ the subspace of $\mathcal L$ formed by functions with zero mean value with respect to $\mu$. It is sufficient to show that any bounded sequence of smooth functions in $\mathcal L_0$ is also bounded in $W^{1,2}(M,\mathit{Vol}_g)$. More precisely, we claim that there exists a constant $C$ such that for any smooth function $u\in\mathcal L_0$ the inequality
$$
\int_M u^2 d\mathit{Vol}_g\leqslant C\cdot\int_M\abs{\nabla u}^2 d\mathit{Vol}_g
$$
holds. Indeed, suppose the contrary. Then there exists a sequence $(u_n)$ such that
\begin{equation}
\label{cont}
\int_M u_n^2 d\mathit{Vol}_g=1,\quad\text{and}\quad\int_M\abs{\nabla u_n}^2 d\mathit{Vol}_g\rightarrow 0. 
\end{equation}
Since the first eigenvalue does not vanish, we also have
\begin{equation}
\label{lambda}
\int_M u^2 d\mu\leqslant\lambda_1^{-1}(\mu,c)\cdot\int_M\abs{\nabla u}^2 d\mathit{Vol}_g
\end{equation}
for any $u\in\mathcal L_0$. Then, after a selection of a subsequence, the  $u_n$'s converge weakly in $\mathcal L_0$, and also strongly in $L_2(M,\mathit{Vol}_g)$, to some function $v\in\mathcal L_0$. By the second relation in~\eqref{cont} this limit function has to be constant almost everywhere with respect to $\mathit{Vol}_g$. Further, relation~\eqref{lambda} shows that $v$ vanishes almost everywhere with respect to the measure $\mu$. Since $\mu$ is not completely singular, then from the above we conclude that $v$ vanishes almost everywhere also with respect to $\mathit{Vol}_g$. However, from~\eqref{cont} we see that the $L_2$-norm of $v$ equals  one. Thus, we arrive at a contradiction, and the claim is proved.
\end{proof}
By the results of Maz'ja the compactness of the embedding $W^{1,2}(M,\mathit{Vol}_g)$ into the space $L_2(M,\mu)$ is characterised by the decay of the isocapacity constant on small balls. More precisely, the following result is essentially contained in~\cite{Ma}, see also~\cite[Sect.~7]{AH}.
\begin{MT1}
Let $\mu$ be a Radon measure supported in a bounded domain $\Omega\subset\mathbf R^2$ with smooth boundary. Then the embedding  $W^{1,2}(\Omega,\mathit{Vol}_g)$ into  $L_2(\Omega,\mu)$ is compact if and only if $\sup_x\beta(B(x,r),\mu)\to 0$ as $r\to 0$, where the supremum is taken over $x\in\Omega$.
\end{MT1}
Combining this result with Lemma~\ref{re2ma}, we obtain the following consequence.
\begin{corollary}
\label{cor:cap}
Under the hypotheses of Lemma~\ref{re2ma}, we have
\begin{itemize}
\item[(i)] if the embedding~\eqref{ce} is compact, then
\begin{equation}
\label{cap_decay}
\sup_{x\in M}\beta(B(x,r),\mu)\longrightarrow 0\qquad r\to 0;
\end{equation}
\item[(ii)] if the measure $\mu$ is not completely singular, has positive eigenvalue, and satisfies~\eqref{cap_decay}, then the embedding~\eqref{ce} is compact.
\end{itemize}
\end{corollary}
\begin{remark*}
First, mention that due to~\eqref{MaIn} the decay hypothesis on the isocapacity constant is equivalent to the growth of the fundamental tone on small balls. Second, following Maz'ja~\cite{Ma}, one can also consider the {\em isocapacity function} $\beta_r(\Omega)$, defined as the quantity
$$
\sup\left\{\mu(F)/\CAP(F,\Omega) : F\subset\Omega\text{ is a compact set, }\mathit{diam}(F)\leqslant r\right\}.
$$
Then the hypothesis~\eqref{cap_decay} in the corollary above can be replaced by the supposition that $M$ can be covered by open sets $\Omega_i$ whose isocapacity functions $\beta_r(\Omega_i)$ converge to zero as $r\to 0$.
\end{remark*}
Recall that by Prop.~\ref{ex:eigen}, the compactness of the embedding~\eqref{ce}  for a measure $\mu$ implies that its first eigenvalue $\lambda_1(\mu, c)$ does not vanish. However, the {\em converse does not hold}. More precisely, by Corollary~\ref{cor:cap} the measures for which the embedding~\eqref{ce} is compact satisfy the following (weaker than~\eqref{cap_decay}) hypothesis
\begin{equation}
\label{cap_ball}
\sup_{x\in M}\mu(B(x,r))\ln(1/r)\rightarrow 0\qquad\text{as~ }r\to 0.
\end{equation}
We claim that there are measures with positive first eigenvalues for which this hypothesis fails. For this it is sufficient to construct a compactly supported measure in $\mathbf R^2$ with bounded logarithmic potential such that the quantity  $\mu(B(x,r))\ln(1/r)$ does not converge to zero uniformly. The boundedness of the potential implies that the isocapacity constant $\beta(\Omega, \mu)$ is bounded as $\Omega$ ranges over a certain class of subdomains and, by Cor.~\ref{iff}, one concludes that the first eigenvalue has to be positive. (The details can be communicated on request.)

The next statement says that a slightly stronger decay hypothesis than~\eqref{cap_ball} is often sufficient for the embedding compactness. 
\begin{lemma}
\label{sc}
Let $M$ be a compact surface, possibly with boundary, endowed with a conformal class of Riemannian metrics, and $\mu$ be a Radon measure on $M$. Suppose that $\mu$ is not completely singular, and its values on small balls satisfy the relation:
\begin{equation}
\label{scH}
\sup_{x\in M}\mu(B(x,r))\ln^q(1/r)\rightarrow 0\qquad\text{as~ }r\to 0,
\end{equation}
where $q>1$. Then the embedding~\eqref{ce} is compact and, in particular, the first eigenvalue $\lambda_1(\mu,c)$ is positive.
\end{lemma}
The hypotheses above actually  yield a stronger conclusion: the space $\mathcal L$ in this case embeds compactly into $L_{2q}(M,\mu)$. Conversely, the compact embedding into $L_{2q}(M,\mu)$ implies  relation~\eqref{scH}, under the hypotheses on the measure above. The proof appears at the end of the section; it is based on the following theorem due to Maz'ja, contained in~\cite[Sect.~8.8]{Ma}, see also~\cite[Sect.~7]{AH}. 
\begin{MT2}
Let $\mu$ be a Radon measure supported in a bounded domain $\Omega\subset\mathbf R^2$ with smooth boundary. Then for any $q>1$ the embedding of $W^{1,2}(\Omega,\mathit{Vol}_g)$ into the space $L_{2q}(\Omega,\mu)$ is compact if and only if
the measure satisfies the following decay property
$$
\sup_{x}\mu(B(x,r))\ln^q(1/r)\to 0\quad\text{ as }r\to 0,
$$
where the sup is taken over all $x\in\Omega$. 
\end{MT2}
We proceed with examples illustrating Lemma~\ref{sc}  in action.
\begin{example}
\label{Lp}
Let $\mu$ be an absolutely continuous measure, that is given by the integral
$$
\mu(E)=\int_E fd\mathit{Vol}_g,\qquad\text{where } E\subset M.
$$
Suppose that the density function $f$ is $L_p$-integrable for some $p>1$. Then we claim that relation~\eqref{scH} holds, and by Lemma~\ref{sc} the embedding~\eqref{ce} is compact. Indeed, by Holder's inequality we obtain
$$
\mu(B(x,r))\leqslant\abs{f}_p\cdot\mathit{Vol}_g(B(x,r))^{1/p^*},
$$
where $|f|_p$ denotes for the $L_p$-norm, and $p^*$ is the Holder conjugate to $p$. Now the claim follows from the fact that the volume term behaves like $O(r^{2/p^*})$ when $r$ tends to zero.
\end{example}
\begin{example}
Generalising the example above one can also consider the so-called $\alpha$-uniform measures;
they satisfy the relation
$$
\mu(B(x,r))\leqslant Cr^\alpha\qquad\text{for any }x\in M,
$$
and some positive constants $C$ and $\alpha$. These, for example, include measures that are absolutely continuous with respect to the $s$-dimensional Hausdorff measures $\mu^s$ with densities in $L_p(M,\mu^s)$, where $p>1$. 
Adding such measures to the  one in the example above, we obtain a variety of non-absolutely continuous measures for which the embedding~\eqref{ce} is compact.
\end{example}

\subsection{Proof of Lemma~\ref{sc}}
We start with the following statement.
\begin{claim}
\label{new:claim}
Let $\mu$ be a finite Radon measure supported in a bounded domain $G\subset\mathbf R^2$. Suppose that the values $\mu(B(x,r))\ln^q(1/r)$ are uniformly bounded in $x$ and $0\leqslant r\leqslant 1$. Then there exists a constant $C_1$ such that
\begin{equation}
\label{1:cap}
\mu(F)\leqslant C_1\cdot\CAP (F,\Omega)
\end{equation}
for any $F\subset\Omega\subset G$, where $F$ is a closed set.
\end{claim}
\begin{proof}
First, we introduce another capacity quantity on compact sets $F$ in the Euclidean plane:
$$
\capty(F)=\inf\left\{\int\varphi^2dV+\int\abs{\nabla\varphi}^2dV:\varphi\in C_0^\infty(\mathbf R^2)\text{ and }\varphi\geqslant 1\text{ on }F\right\}.
$$
As is known~\cite{Ma,Re}, its values on balls behave asymptotically like $O(\ln(1/r))$, and by the claim hypotheses we obtain that
$$
\mu(B(x,r))\leqslant C_2\cdot\capty(B(x,r))^q
$$ 
for some constant $C_2$, where $x\in\mathbf R^2$ and $0\leqslant r\leqslant 1$. By the result of Maz'ja in~\cite[Sect.~8.5]{Ma}, this inequality extends to any compact set $F$,
\begin{equation}
\label{ma:cap}
\mu(F)\leqslant C_3\cdot\capty(F)^q,
\end{equation}
possibly with another constant $C_3$ independent of $F$. Now we claim that the latter implies that
\begin{equation}
\label{q:cap}
\mu(F)\leqslant C_4\cdot\CAP (F,\Omega)^q
\end{equation}
for any $F\subset\Omega\subset G$. Indeed, as is known~\cite[Sect.~6]{Re}, there is a constant $C_5$, depending on the diameter of $G$ only, such that
$$
\capty(F)\leqslant C_5\cdot\CAP(F,G)\leqslant C_5\cdot \CAP(F,\Omega)
$$
for any $F\subset\Omega\subset G$, where the second inequality is a monotonicity property of $\CAP$. This together with~\eqref{ma:cap} demonstrates inequality~\eqref{q:cap}, which, in turn, yields inequality~\eqref{1:cap}; the constant $C_1$ can be chosen to be the maximum of $C_4$ and the total mass of $\mu$. 
\end{proof}
To prove Lemma~\ref{sc} we fix a reference metric $g\in c$ and choose a finite open covering $(V_i)$ of $M$ by charts on which $g$ is conformally Euclidean. Using the partition of unity, we can decompose $\mu$ into the sum of measures $\mu_i$, where each $\mu_i$ is supported in $V_i$. By $c_i$ we denote the conformal class on $V_i$ obtained by restricting the metrics from $c$. Combining Claim~\ref{new:claim} and Corollary~\ref{iff}, we see that the first eigenvalues $\lambda_1(\mu_i,c_i)$ are positive. It is straightforward to see that so are the first eigenvalues $\lambda_1(\mu_i,c)$,
$$
\lambda_1(\mu_i,c)\geqslant\lambda_1(\mu_i,c_i)>0.
$$
Now we apply the second Maz'ja theorem together with Lemma~\ref{re2ma} to conclude that the embedding
$$
L_2(M,\mu_i)\cap L_2^1(M,\mathit{Vol}_g)\subset L_2(M,\mu_i)
$$
is compact for any $i$, and hence so is the embedding~\eqref{ce}.
\qed

\section{Weak maximisers for the first eigenvalue}
\label{weakmax:ex}
\subsection{The main theorem}
Recall that, identifying conformal metrics with their volume forms, we extended the eigenvalues $\lambda_k(g)$ to a class of Radon probability measures on $M$. On the {\em class of continuous measures} the eigenvalues are still bounded, and the purpose of this section is to show the $\sup\lambda_1(\mu,c)$ is achieved in this class. More precisely, we have the following statement.
\begin{TB1}
Let $M$ be a compact surface, possibly with boundary, endowed with a conformal class $c$ of Riemannian metrics. Suppose that 
\begin{equation}
\label{hypo1}
\sup\{\lambda_1(\mu,c)\mu(M):\mu\text{ is a continuous Radon measure  on }M\}>8\pi.
\end{equation}
Then any $\lambda_1$-maximising sequence of Radon probability measures contains a subsequence that converges weakly to a continuous Radon measure $\mu$ at which the supremum on the left-hand side is achieved.
\end{TB1}
Before proving the theorem we make two remarks. First, the maximal measure clearly has a positive first eigenvalue and, thus, satisfies a certain isocapacitory inequality, see Sect.~\ref{nonvan}. In particular, the class of continuous Radon measures in the theorem above can be significantly narrowed, for example, to the Radon measures that do not charge sets of zero capacity.  Second, the following result of Colbois and El Soufi~\cite{CEl} shows that the hypothesis~\eqref{hypo1} is not very significant for closed surfaces $M$: for any conformal class $c$ on a closed surface $M$ the quantity 
$$
\sup\{\lambda_1(g)\mathit{Vol}_g(M):g\in c\}
$$
is greater or equal to $8\pi$. 

Due to the upper-semicontinuity property of the eigenvalues the proof of Theorem~$\mathrm{B_1}$ is essentially concerned with ruling out measures with non-trivial discrete part as limit maximal measures. 
\begin{proof}[Proof of Theorem~$\mathrm{B_1}$.]
Denote by $\Lambda_1$ the quantity
$$
\sup\{\lambda_1(\mu,c):\mu\text{ is a continuous Radon probability  measure on }M\},
$$
and let $\mu_n$ be a maximising sequence of continuous Radon measures, $\lambda_1(\mu_n,c)\to \Lambda_1$ as $n\to +\infty$. Since the space of Radon probability measures on a compact surface $M$ is weakly compact, we can assume that the $\mu_n$'s converge weakly to a Radon probability measure $\mu$. By upper semi-continuity (Lemma~\ref{up_semi}), for a proof of the theorem it is sufficient to show that $\mu$ is continuous. Since $\Lambda_1>8\pi$, then by Lemma~\ref{lambda1c} below the measure $\mu$ can not be a Dirac measure. Further, the combination of upper semi-continuity and Lemma~\ref{comp} implies that $\mu$ can not have a non-trivial discrete part and, thus, is a continuous Radon measure.
\end{proof}

\subsection{Concentration of measures}
Recall that by the example in Sect.~\ref{prems} the first eigenvalue of the Dirac measure is infinite. Nevertheless, the following lemma shows that it is possible to bound the $\lim\sup\lambda_1(\mu_n)$ for a sequence $\mu_n$ converging to the Dirac measure. A similar statement for Riemannian volume measures has been sketched in~\cite[p.~888-889]{Na96}, and the details have been worked out in~\cite{Gi}; we give a proof following the idea in~\cite{Ko}.
\begin{lemma}
\label{lambda1c}
Let $(M,c)$ be a compact Riemann surface, possibly with boundary, and $\mu_n$ be a sequence of continuous Radon probability measures converging weakly to the Dirac measure $\delta_x$, $x\in M$. Then $\lim\sup\lambda_1(\mu_n)$ is not greater than $8\pi$.
\end{lemma}
\begin{proof}
First, if $M$ has a boundary, then it can be viewed as a subdomain of another Riemannian surface. Thus, without loss of generality we may assume that $x$ is an interior point. Let $\Omega$ be an open coordinate ball around $x\in M$ on which the metric $g$ is conformally Euclidean, and let
$$
\phi:\Omega\rightarrow S^2\subset\mathbf R^3
$$
be a conformal map into the unit sphere in $\mathbf R^3$. Since a point on Euclidean plane has zero capacity, then for any $\varepsilon>0$ there exists a function $\psi\in C^\infty_0(\Omega)$ such that $0\leqslant\psi\leqslant 1$,
$$
\psi=1\text{ in a neighbourhood of }x,\quad\text{and}\quad\int_M\abs{\nabla\psi}^2\mathit{dVol}g<\varepsilon.
$$
By the Hersch lemma, Appendix~\ref{ap:gy}, there exists a conformal transformation $s_n:S^2\to S^2$ such that
$$
\int_M\psi(x^i\circ s_n\circ\phi)d\mu_n=0\quad\text{for any}\quad i=1,2,3,
$$
where $(x^i)$ are coordinate functions in $\mathbf R^3$. Using the functions $\varphi_n^i=\psi(x^i\circ s_n\circ\phi)$ as test-functions for the Rayleigh quotient, we obtain
$$
\lambda_1(\mu_n,c)\int_M(\varphi_n^i)^2d\mu_n\leqslant\int_M\abs{\nabla\varphi_n^i}^2\mathit{dVol}_g
$$
for any $i=1,2,3$. Summing over all $\imath$'s yields
\begin{equation}
\label{RT}
\lambda_1(\mu_n,c)\int_M\psi^2d\mu_n\leqslant\sum_i\int_M\abs{\nabla\varphi_n^i}^2\mathit{dVol}_g.
\end{equation}
The right-hand side can be estimated as
\begin{multline*}
\sum_i\int_M\abs{\nabla\varphi^i_n}^2\mathit{dVol}_g\leqslant\sum_i\int_M\psi^2\abs{\nabla(x^i\circ s_n\circ\phi)}^2\mathit{dVol}_g\\
+2\sum_i\int_M\psi\abs{\nabla(x^i\circ s_n\circ\phi)}\abs{\nabla\psi}\mathit{dVol}_g+\int_M\abs{\nabla\psi}^2\mathit{dVol}_g.
\end{multline*}
The first sum on the right-hand side can be further estimated by the quantity
$$
\sum_i\int_\Omega\abs{\nabla(x^i\circ s_n\circ\phi)}^2\mathit{dVol}_g\leqslant\sum_i\int_{S^2}\abs{\nabla(x^i\circ s_n)}^2\mathit{dVol}_{S^2}=8\pi;
$$
here we used the conformal invariance of the Dirichlet energy, which in particular implies that the energy of a conformal diffeomorphism of $S^2$ equals $8\pi$. Similarly the second sum is not greater that
\begin{multline*}
2\sum_i\int_\Omega\abs{\nabla(x^i\circ s_n\circ\phi)}\abs{\nabla\psi}\mathit{dVol}_g\leqslant 2\varepsilon^{1/2}\sum_i\left(\int_\Omega\abs{\nabla(x^i\circ s_n\circ\phi)}^2\mathit{dVol}_g\right)^{1/2}\\
\leqslant 10\pi^{1/2}\varepsilon^{1/2}.
\end{multline*}
Using these two estimates and the fact that the Dirichlet energy of $\psi$ is less than $\varepsilon$, we obtain
$$
\sum_i\int_M\abs{\nabla\varphi^i_n}^2\mathit{dVol}_g\leqslant 8\pi+10\pi^{1/2}\varepsilon^{1/2}+\varepsilon.
$$
Combining the last inequality with~\eqref{RT}, and passing to the limit as $n\to +\infty$, we arrive at the following relation
$$
\lim\sup\lambda_1(\mu_n,c)\leqslant 8\pi+10\pi^{1/2}\varepsilon^{1/2}+\varepsilon.
$$
Since $\varepsilon>0$ is arbitrary, we conclude that the left-hand side is not greater than $8\pi$.
\end{proof}
\begin{remark*}
There is a version of Lemma~\ref{lambda1c} also for higher eigenvalues. More precisely, the arguments outlined in Appendix~\ref{ap:gy} yield the following statement: for any sequence of Radon measures $(\mu_n)$ converging weakly to a pure discrete measure  the inequality
$$
\lim\sup\lambda_k(\mu_n,c)\leqslant C_*k
$$
holds, where $C_*$ is the universal Korevaar(-Grigor'yan-Yau) constant. 
\end{remark*}

\section{Elements of regularity theory}
\label{em1}
\subsection{The main theorem}
Let $(M,c)$ be a compact Riemann surface. For a given Radon probability measure $\mu$ on $M$ by its {\it conformal deformation}  we call the family of probability measures
\begin{equation}
\label{deform}
\mu_t(X)=\left(\int_Xe^{\phi t}d\mu\right)/\left(\int_Me^{\phi t}d\mu\right),
\end{equation}
where $X\subset M$ is a Borel subset, and $\phi\in L^\infty(M)$ is a {\em generating function}. Clearly, any two generating functions that differ by a constant define the same family $\mu_t$. Thus, it is sufficient to consider generating functions $\phi$ that have zero mean-value with respect to $\mu$. This assumption is made throughout the rest of the paper.
\begin{defin}
\label{def:ex}
A Radon probability measure $\mu$ on a compact Riemann surface $(M,c)$ is called {\em extremal} for the $k$th eigenvalue $\lambda_k(\mu,c)$ if for any $\phi\in L^\infty(M)$ the function $\lambda_k(\mu_t,c)$, where
$\mu_t$ is defined by~\eqref{deform}, satisfies either the inequality
$$
\lambda_k(\mu_t,c)\leqslant\lambda_k(\mu,c)+o(t)\quad\text{ as } t\to 0,
$$
or the inequality
$$
\lambda_k(\mu_t,c)\geqslant\lambda_k(\mu,c)+o(t)\quad\text{ as } t\to 0.
$$
\end{defin}
In particular, we see that any $\lambda_k$-{\em maximiser} is extremal under conformal deformations. The definition above is a natural generalisation of the one given by Nadirashvili~\cite{Na96}, and also studied in~\cite{El00,El03}, for smooth Riemannian metrics. 

The purpose of this section is to study regularity properties of extremal measures. Recall that any Radon measure $\mu$ decomposes into the sum
$$
\mu=\int f\mathit{dVol}_g+\mu\lfloor\Sigma
$$
of its absolutely continuous and singular parts; the set $\Sigma$ has zero Lebesgue measure and is called the {\em singular set} of $\mu$. This decomposition motivates the terminology used in the sequel: we say that a measure $\mu$ "defines a metric conformal to $g$ away from the singular set $\Sigma$", viewing the density function $f$ as the "conformal factor of such a metric". The regularity properties of a measure $\mu$ are essentially concerned with the following questions.
\begin{itemize}
\item[(i)] How smooth is the density function $f$ of a given extremal measure? When is it $C^\infty$-smooth?
\item[(ii)] What are the properties of the singular set $\Sigma$ of an extremal measure, and when is it empty?
\end{itemize} 
Below we give complete answers to these questions under the hypothesis that the embedding
\begin{equation}
\label{emb}
\mathcal{L}=L_2(M,\mu)\cap L_2^1(M,\mathit{Vol}_g)\subset L_2(M,\mu).
\end{equation}
is compact. We refer to Sect.~\ref{nonvan} for the examples and description of measures that satisfy this hypothesis.

Another question, closely related to regularity, is concerned with the properties of the support $S$ of a given $\lambda_k$-extremal measure. For example, if  a $\lambda_k$-maximal measure is the limit of Riemannian volume measures, then the regions where it vanishes are precisely the regions where the corresponding Riemannian metrics collapse. In general, the support of a $\lambda_k$-extremal measure does not have to coincide with $M$. More precisely, the examples below show that there are completely singular extremal measures, that is, supported in zero Lebesgue measure sets.
\begin{example}[Singular $\lambda_1$-extremal measure on a disk]
\label{ex:disk}
Let $M$ be a $2$-dimensional disk, and $\mu_g$ be a boundary length measure of the Euclidean metric $g$. Rescaling the metric, we can suppose that $\mu_g$ is a probability measure. Its first eigenvalue $\lambda_1(\mu_g,[g])$ coincides with the first Steklov eigenvalue of $g$ and, as is known~\cite{Wei,FS}, is equal to $2\pi$. Moreover, the argument in~\cite[Th.~2.3]{FS} shows that $\mu_g$ maximises $\lambda_1(\mu,[g])$ among all continuous probability measures supported in the boundary $\partial M$. Since the conformal deformations given by~\eqref{deform} do not change the support of a measure, we conclude that $\mu_g$ is $\lambda_1$-extremal in the sense of Definition~\ref{def:ex}. 
\end{example}
\begin{example}[Singular $\lambda_1$-extremal measure on a sphere]
Let $M$ be a $2$-dimensional sphere, $E$ be its equator, and $M^+$ be a hemisphere whose boundary is $E$. For any continuous probability measure $\mu$ supported in $E$ it is straightforward to show that
$$
\lambda_1(\mu,[g_R])=2\lambda_1(\mu,[g^+_R]),
$$
where $g_R$ and $g^+_R$ denote the round metrics on $M$ and $M^+$ respectively. Let $\mu_R$ be a length measure on the equator $E$ corresponding to the round metric on $M$; we may assume that it is rescaled to be a probability measure. Using the result in Example~\ref{ex:disk}, it then follows that $\mu_R$ maximises $\lambda_1$ on $M$ among all continuous probability measures supported in the equator $E$. Since the conformal deformations given by~\eqref{deform} do not change the support of a measure, as in Example~\ref{ex:disk}, we conclude that $\mu_R$ is $\lambda_1$-extremal in the sense of Defintion~\ref{def:ex}.
\end{example}

Now we state our principal result; it deals with regularity properties of a $\lambda_k$-extremal measure in the interior of its support.
\begin{TCk}
Let $M$ be a compact surface, possibly with boundary, endowed with a conformal class $c$ of Riemannian metrics. Let $\mu$ be a $\lambda_k$-extremal measure  which is not completely singular and such that the embedding~\eqref{emb} is compact. 
\begin{itemize}
\item [(i)] Then the measure $\mu$ is absolutely continuous (with respect to $\mathit{Vol}_g$, $g\in c$) in the interior of its support $S\subset M$, its density function is $C^\infty$-smooth in $S$ and vanishes at isolated points only. In other words, the measure $\mu$ defines a $C^\infty$-smooth metric on $S$, conformal to $g\in c$ away from isolated degeneracies which are conical singularities.
\item [(ii)] If the support of the measure $\mu$ does not coincide with $M$, then the measure has a non-trivial singular set $\Sigma\subset M\backslash\INT S$.
\end{itemize}
\end{TCk}
The following example suggests that the compact embedding hypothesis may hold when an extremal metric has sufficiently many symmetries.
\begin{example}[Symmetries and regularity]
Let $M$ be a surface, possibly with boundary, and $c$ be a conformal class of Riemannian metrics on it. Further, let $\mu$ be a $\lambda_k$-extremal metric on $M$, understood as a non-completely singular Radon measure, and suppose that $\mu$ is invariant under a free smooth circle action on $M$. Then $\mu$ is a Riemannian metric which is $C^\infty$-smooth in the interior of its support. Indeed, by the classical disintegration theory~\cite{DM} any circle-invariant measure locally splits as a product of two measures; one of them is a uniform measure on a reference orbit, see details in~\cite{Ko2}. This shows that there is a constant $C$ such that for any sufficiently small ball $B(x,r)\subset M$ the following inequality holds
$$
\mu(B(x,r))\leqslant Cr\qquad\text{for any }x\in M.
$$
Now Lemma~\ref{sc} implies that the embedding~\eqref{emb} is compact, and by Theorem~$C_k$ the measure $\mu$ is the volume measure of a $C^\infty$-smooth metric in the interior of its support.
\end{example}
We end this introduction with remarks on conical singularities of extremal metrics. Recall that for a given metric a point $p\in M$ is called its conical singularity of order $\alpha$ (or of angle $2\pi(\alpha+1)$) if in an appropriate local complex coordinate the metric has the form $\abs{z}^{2\alpha}\rho(z)\abs{dz}^2$, where $\rho(z)>0$. In other words, near $p$ the metric is conformal to the Euclidean cone of total angle $2\pi(\alpha+1)$. First, the conical singularities of an extremal metric in Theorem~$C_k$ have angles that are integer multiples of $2\pi$. This follows from the proof, where we show that they correspond to branch points of certain harmonic maps. The above applies to singularities in the interior of the supports only. Mention that on the boundary an extremal metric can have more complicated degeneracies. For example, the metric on a $2$-dimensional disk $D$, regarded as a punctured round sphere, maximises the first eigenvalue and vanishes on the boundary.
\begin{example}[Smoothness of conical singularities]
Let $g$ be a metric with conical singularities and unit volume on $M$. Suppose that it is $\lambda_k$-extremal under conformal deformations, that is in the sense of Definition~\ref{def:ex}. We claim that such a metric has to be $C^\infty$-smooth, and the angles at its conical singularities are integer multiplies of $2\pi$. Indeed, by Example~\ref{Lp} the embedding~\eqref{emb} is compact, and the statement follows from Theorem~$C_k$ together with the discussion above. Mention that the hypothetical $\lambda_1$-maximal metric on a genus $2$-surface, obtained in~\cite{JLNNP}, satisfies this conclusion.
\end{example}
\begin{example}[Extremal absolutely continuous measures]
Let $\mu$ be an absolutely continuous probability measure on $M$, whose density function is $L^p$-integrable, where $p>1$; see Example~\ref{Lp}. One can view $\mu$ as the volume measure of a metric conformal to a genuine Riemannian metric on $M$ whose conformal factor is $L^p$-integrable. Such singular metrics naturally occur on Alexandrov surfaces of bounded integral curvature, see~\cite{Ko3}. Suppose that $\mu$ is $\lambda_k$-extremal under conformal deformations. Then by Theorem~$C_k$ the support of $\mu$ coincides with the whole surface $M$, and the density function is $C^\infty$-smooth everywhere on $M$.
\end{example}

\subsection{Continuity properties}
We start with establishing the continuity properties of eigenvalues and eigenspaces corresponding to the family of measures $\mu_t$. We consider these issues in a slightly more general setting that is necessary for applications, describing a suitable topology on the space of probability measures.
\begin{defin}
By the {\it integral distance} between two probability measures $\mu$
and $\mu'$, we call the quantity
$$
d(\mu,\mu')=\sup_{v\geqslant 0}\abs{\ln\left(\int vd\mu/\int vd\mu'\right)},
$$
where the supremums are taken over {\it non-trivial} continuous functions on $M$.
\end{defin}
In general, the distance $d(\mu,\mu')$ may take infinite values; however, it does determine a topology on the space of probability measures, which is stronger than the weak topology. For example, the family of measures $\mu_t$ given by~\eqref{deform}, is always continuous in it. Mention that for measures with finite distance the corresponding $L_2$-spaces, regarded as topological vector spaces, coincide. In particular, the embedding~\eqref{emb} is compact or not for such measures simultaneously. In the sequel we often use the introduced distance in the form of the following inequality:
$$
\abs{1-\left(\int vd\mu/\int vd\mu'\right)}\leqslant\delta(\mu,\mu'):=\exp{d(\mu,\mu')}-1,
$$
where $v$ is an arbitrary non-negative function. We demonstrate this in the following lemma.
\begin{lemma}
\label{c:eiv}
Let $(M,c)$ be a compact Riemann surface, possibly with boundary, and $\mu$ be a probability measure on $M$ whose eigenvalue $\lambda_k(\mu,c)$ is finite. Then for any sequence $(\mu_n)$ of
probability measures that converge in the integral distance to $\mu$, we have
$$
\lambda_k(\mu_n,c)\longrightarrow\lambda_k(\mu,c)\qquad\text{as~~}n\to +\infty.
$$
\end{lemma}
\begin{proof}
First, in view of the upper semi-continuity property (Prop.~\ref{up_semi}), it is sufficient to prove that
\begin{equation}
\lambda_k(\mu,c)\leqslant\lim\inf\lambda_k(\mu_n,c).
\end{equation}
Let $\Lambda_n$ be a $(k+1)$-dimensional space such that
$$
\sup_{u\in\Lambda_n}\matheur R_c(u,\mu_n)\leqslant\lambda_k(\mu_n,c)+1/n.
$$
We claim that the sequence
\begin{equation}
\label{sup_dif}
\sup_{\Lambda_n}\matheur R_c(u,\mu_n)-\sup_{\Lambda_n}\matheur R_c(u,\mu)
\end{equation}
converges to zero as $n\to +\infty$. Indeed, for any $u\in\Lambda_n$,
we have
\begin{multline*}
~~~~~\abs{\matheur R_c(u,\mu_n)-\matheur R_c(u,\mu)}\leqslant\delta(\mu,\mu_n)
\matheur R_c(u,\mu_n)\leqslant \delta(\mu,\mu_n)(\lambda_k(\mu_n,c)+1/n)\\
\leqslant C\cdot \delta(\mu,\mu_n).~~~~~
\end{multline*}
Here the first inequality follows by the definition of the quantity $\delta(\mu,\mu_n)$, and the constant $C$ is an upper bound for the sequence $(\lambda_k(\mu_n,c)+1/n)$. Since $\lambda_k(\mu,c)$ is finite, by upper semi-continuity such a bound exists. The last estimate shows that the absolute value of  quanity~\eqref{sup_dif} is also bounded by $C\cdot\delta(\mu,\mu_n)$, and hence converges to zero. Thus, we have
$$
\lambda_k(\mu,c)\leqslant\lim\inf (\sup_{\Lambda_n}\matheur R_c(u,\mu))=\lim\inf (\sup_{\Lambda_n}\matheur R_c(u,\mu_n))=\lim\inf\lambda_k(\mu_n,c),
$$
and the claim is demonstrated.
\end{proof}

We proceed with the continuity properties of eigenspaces. Below we suppose that for Radon measures  $\mu$ and $\mu_n$ the embedding~\eqref{emb} is compact. Denote by $E_k$ and $E_{n,k}$ the eigenspaces corresponding to $\lambda_k(\mu,c)$ and $\lambda_k(\mu_n,c)$ respectively, and by $\Pi_k$ and $\Pi_{n,k}$ the orthogonal projections on $E_k$ and $E_{n,k}$, regarded as subspaces in  $L_2(M,\mu)$. The following lemma can be obtained as a consequence of Kato's perturbation theory for Dirichlet forms~\cite{Kato}; the proof details can be found in Appendix~\ref{ap:proofs}.
\begin{lemma}
\label{c:eif}
Let $(M,c)$ be a compact Riemann surface, possibly with boundary, and let $(\mu_n)$ be a sequence of Radon probability measures converging in the integral distance to a Radon measure $\mu$.  Then the eigenspace projections $\Pi_{n,k}$ converge to the projection $\Pi_k$ in the norm topology as operators in $L_2(M,\mu)$.
\end{lemma}
\begin{remark*}
The arguments in Appendix~\ref{ap:proofs} show that the lemma above can be re-phrased in a number of other ways. For example, if $\Pi^*_{n,k}$ is an orthogonal projection on $E_{n,k}$ as a subspace in $L_2(M,\mu_n)$, then the norm $|\Pi_k-\Pi^*_{n,k}|$ of the operators in $L_2(M,\mu_n)$ also converges to zero as $n\to +\infty$.
\end{remark*}

\subsection{First variation formulas}
For a zero mean-value function $\phi\in L^\infty(M)$ by $L_\phi(u,\mu)$ we denote the quotient
$$
-\matheur R_c(u,\mu)\cdot\left(\int_Mu^2\phi d\mu\right)/\left(\int_Mu^2d\mu\right).
$$
The purpose of this sub-section is to prove the following first variation formulas for the eigenvalue functionals.
\begin{lemma}
\label{d:eiv}
Let $(M,c)$ be a compact Riemann surface, possibly with boundary, and $\mu$ be a Radon probability measure on $M$ such that the embedding~\eqref{emb} is compact. Then for any family of measures $\mu_t$, generated by a zero mean-value $\phi\in L^\infty(M)$, the function $\lambda_k(\mu_t,c)$ has left and right derivatives which satisfy the relations
$$
\left.\frac{d}{dt}\right|_{t=0-}\lambda_k(\mu_t,c)=
\sup_{u\in E_k}L_\phi(u,\mu),
$$
$$
\left.\frac{d}{dt}\right|_{t=0+}\lambda_k(\mu_t,c)=
\inf_{u\in E_k}L_\phi(u,\mu),
$$
where $E_k$ is the space spanned by eigenfunctions corresponding to the eigenvalue $\lambda_k(\mu,c)$, and the sup and inf are taken over non-trivial functions.
\end{lemma}
\begin{proof}
Below we prove the second identity. The first identity follows by similar arguments. Let $E_{k,t}$ and $E_k$ be the eigenspaces corresponding to $\lambda_k(\mu_t,c)$ and $\lambda_k(\mu,c)$.The following statements are proved  in Appendix~\ref{ap:proofs}.
\begin{claim}
\label{c1}
The eigenvalues $\lambda_k(\mu_t)$ and $\lambda_k(\mu)$ satisfy the
following inequalities:
$$
\lambda_k(\mu_t)\leqslant\inf_{u\in E_k}\matheur
R_c(u,\mu_t)+o(t)\qquad\text{as}\quad t\to 0,
$$
$$
\lambda_k(\mu)\leqslant\inf_{u\in E_{k,t}}\matheur
R_c(u,\mu)+o(t)\qquad\text{as}\quad t\to 0,
$$
where the infimums are taken over non-trivial functions.
\end{claim}
\begin{claim}
\label{c2}
The following limit identities hold:
$$
\inf_{E_{k,t}}L_\phi(u,\mu)\longrightarrow\inf_{E_k}L_\phi(u,\mu)\qquad
\text{as}\quad t\to 0,
$$
$$
\sup_{E_{k,t}}L_\phi(u,\mu)\longrightarrow\sup_{E_k}L_\phi(u,\mu)\qquad
\text{as}\quad t\to 0,
$$
where the infimums and supremums are assumed to be
taken over non-trivial functions $u$.
\end{claim}

\noindent
First, it is straightforward to see from the definition of $\mu_t$ that for any $u\in\mathcal L$ the following relation holds:
$$
\abs{\frac{1}{t}\left(\int_Mu^2d\mu_t-\int_Mu^2d\mu\right)-\int_Mu^2\phi d\mu}\leqslant\varepsilon(t)
\cdot\int_Mu^2d\mu,
$$
where $\varepsilon(t)$ is a quantity that does not depend on $u$ and converges to zero as $t\to 0$.
A further computation yields
\begin{equation}
\label{kineq:1}
\abs{\frac{1}{t}(\matheur R_c(u,\mu_t)-\matheur R_c(u,\mu))-L_\phi(u,\mu)}\leqslant \matheur R_c(u,\mu_t)\cdot (\delta(\mu,\mu_t)\abs{\phi}_\infty+\varepsilon(t))
\end{equation}
for any function $u\in\mathcal L$. Evaluating the quantities in this inequality on $u\in E_k$, we conclude that
$$
\frac{1}{t}(\inf_{u\in E_k}\matheur R_c(u,\mu_t)-\lambda_k(\mu))\longrightarrow\inf_{u\in E_k}L_\phi(u,\mu)
\qquad\text{as}\quad t\to 0+.
$$
Combining this with the first relation in Claim~\ref{c1}, we get 
\begin{equation}
\label{sup_ineq}
\lim\sup_{t\to 0+}\frac{1}{t}(\lambda_k(\mu_t)-\lambda_k(\mu))\leqslant \inf_{u\in E_k}L_\phi(u,\mu).
\end{equation}
Now evaluating the quantities in inequality~\eqref{kineq:1} on $u\in E_{k,t}$, we obtain that
$$
\inf_{u\in E_{k,t}}L_\phi(u,\mu)-\frac{1}{t}(\lambda_k(\mu_t)-\inf_{u\in E_{k,t}}\matheur R_c(u,\mu))\longrightarrow 0
\qquad\text{as}\quad t\to 0+.
$$ 
Combining this with the second relation in Claim~\ref{c1}, we conclude that
\begin{equation}
\label{inf_ineq}
\lim\inf_{t\to 0+}(\inf_{u\in E_{k,t}}L_\phi(u,\mu))\leqslant\lim\inf_{t\to 0+}\frac{1}{t}(\lambda_k(\mu_t)-\lambda_k(\mu)).
\end{equation}
Now by Claim~\ref{c2} the quantity on the left-hand side above coincides with $\inf_{E_k}L_\phi(u,\mu)$, and the second identity of the lemma follows by combination of inequalities~\eqref{sup_ineq}
and~\eqref{inf_ineq}.
\end{proof}

\subsection{Proof of Theorem~$C_k$}
The following lemma is a key ingredient in our approach to the regularity theory for extremal measures. It is a sharpened version of the statement originally discovered by Nadirashvili~\cite{Na96} for Riemannian metrics.
\begin{lemma}
\label{nadir}
Let $(M,c)$ be a compact Riemannian surface, possibly with boundary, and $\mu$ be a Radon probability measure on $M$ such that the embedding~\eqref{emb} is compact. Then the following hypotheses are equivalent:
\begin{itemize}
\item[(i)] the measure $\mu$ is $\lambda_k$-extremal;
\item[(ii)] the quadratic form
$$
u\longmapsto\int_Mu^2\phi d\mu
$$
is indefinite on the eigenspace $E_k$ for any zero mean-value function $\phi\in L^\infty(M)$;
\item[(iii)] there exists a finite collection of $\lambda_k$-eigenfunctions $(u_i)$ such that $\sum_i u_i^2=1$ on the support of $\mu$.
\end{itemize}
\end{lemma}
\begin{proof}
The equivalence of the first two statements is a direct consequence of Lemma~\ref{d:eiv}. Indeed, since the left and right derivatives of $\lambda_k(\mu_t,c)$ exist, the $\lambda_k$-extremality is equivalent
to the relation
$$
\left.\frac{d}{dt}\right|_{t=0+}\lambda_k(\mu_t,c)\cdot\left.\frac{d}{dt}\right|_{t=0-}\lambda_k(\mu_t,c)\leqslant 0
$$
for any conformal deformation $\mu_t$. Using the formulas for the derivatives, we conclude that $\mu$ is $\lambda_k$-extremal if and only if the form $L_\phi(u,\mu)$ is indefinite on $E_k$ for any zero mean-value function $\phi\in L^\infty(M)$. The latter is equivalent to the hypothesis~$(ii)$.

\noindent
$(ii)\Rightarrow (iii)$. Let $K\subset L_1(M,\mu)$ be the convex hull of the set of squared $\lambda_k$-functions $\{u^2:$ $u\in E_k\}$. Suppose the contrary to the hypotheses~$(iii)$; then $1\ne K$. By classical separation results, there exists a function $\psi\in L^\infty(M)$ such that
$$
\int_M1\cdot\psi d\mu<0\quad\text{and}\quad \int_M q\cdot\psi d\mu>0,\quad\text{for any }q\in K\backslash\{0\}.
$$
Let $\psi_0$ be the mean-value part of $\psi$,
$$
\psi_0=\psi-\int_M\psi d\mu.
$$
Then for any eigenfunction $u\in E_k$ we have
$$
\int_Mu^2\psi_0 d\mu=\int_Mu^2\psi d\mu-\left(\int_M\psi d\mu\right)\left(\int_Mu^2 d\mu\right)>0.
$$
This is a contradiction with~$(ii)$.

\noindent
$(iii)\Rightarrow (ii)$. Conversely, let $(u_i)$ be a finite collection of eigenfunctions satisfying the hypothesis~$(iii)$. Then for any $\phi\in L^\infty(M)$ with zero mean-value, we have
$$
\int_M(\sum_iu_i^2)\phi d\mu=\int_M\phi d\mu=0.
$$
This demonstrates the hypothesis~$(ii)$.
\end{proof}

\begin{proof}[Proof of Theorem~$\mathrm{C}_k$: part~(i)]
Let $(u_i)$, where $i=1,\ldots,\ell$, be a collection of eigenfunctions from Lemma~\ref{nadir}. By Prop.~\ref{prop:eigen} they satisfy the integral identity
\begin{equation}
\label{neq:eigen}
\int_M\langle\nabla u_i,\nabla\varphi\rangle\mathit{dVol}_g=\lambda_k(\mu,c)\int_Mu_i\cdot\varphi d\mu
\end{equation}
for any function $\varphi\in\mathcal L$. Let $S\subset M$ be the support of an extremal measure $\mu$; we suppose that its interior is not empty. Taking $\varphi$ to be $u_i\cdot\psi$, where $\psi\in C_0^\infty(S)$, we can re-write relation~\eqref{neq:eigen} in the form
$$
\int_M\abs{\nabla u_i}^2\psi\mathit{dVol}_g+\frac{1}{2}\int_M\langle\nabla (u_i^2),\nabla\psi\rangle\mathit{dVol}_g=\lambda_k(\mu,c)\int_Mu_i^2\psi d\mu.
$$
Summing up and using the relation $\sum_iu_i^2=1$ on $S$, we obtain
$$
\int_S\left(\sum_i\abs{\nabla u_i}^2\right)\psi\mathit{dVol}_g=\lambda_k(\mu,c)\int_S\psi d\mu
$$
for any compactly supported smooth function $\psi$. This implies that the measure $\mu$ is absolutely continuous with respect to $\mathit{Vol}_g$ in the interior of $S$, and its density function has the form
\begin{equation}
\label{density}
\left(\sum_i\abs{\nabla u_i}^2\right)/\lambda_k(\mu,c).
\end{equation}
Now equation~\eqref{neq:eigen} can be re-written in the form
$$
\int_S\langle\nabla u_i,\nabla\varphi\rangle\mathit{dVol}_g=\int_S\left(\sum_i\abs{\nabla u_i}^2\right)u_i\varphi\mathit{dVol}_g,
$$
where $\varphi$ is a smooth function supported in $S$. This relation is precisely the equation on a map
\begin{equation}
\label{map}
U:M\supset \INT S\ni x\longmapsto (u_1(x),\ldots, u_\ell(x))\in
S^{\ell-1}\subset\mathbf R^\ell
\end{equation}
to be {\em weakly harmonic} with respect to the standard round metric on $S^{\ell-1}$, and by Helein's regularity theory~\cite{He} we conclude that the map given by~\eqref{map} is $C^\infty$-smooth. The zeroes of the density function~\eqref{density} correspond to the branch points of the harmonic map $U$ and, as is known~\cite{Jo,Sa}, are isolated. As a branch point such a zero has a well-defined order, that is in an appropriate local complex coordinate near it the density $\abs{\nabla U}^2$ has the form $z^{2l}\rho(z)$, where $\rho(z)>0$ and $l\geqslant 1$ is an integer.
\end{proof}
\begin{proof}[Proof of Theorem~$\mathrm{C}_k$: part~(ii)]
Setting $\varphi$ to be equal to $u_i$ in relation~\eqref{neq:eigen}, and summing over the $i$'s, we obtain
$$
\int_M\left(\sum_i\abs{\nabla u_i}^2\right)\mathit{dVol}_g=\lambda_k\int_M\left(\sum u_i^2\right)d\mu.
$$
Since, by Lemma~\ref{nadir}, the sum $\sum u_i^2$ equals one on the support $S$ of the measure $\mu$, we obtain
\begin{equation}
\label{id:1}
\int_M\left(\sum_i\abs{\nabla u_i}^2\right)\mathit{dVol}_g=\lambda_k.
\end{equation}
On the other hand, the absolutely continuous part $\mu_{\mathit{abs}}$ of $\mu$ has the form~\eqref{density} in the interior of $S$, and 
\begin{equation}
\label{id:2}
\mu_\mathit{abs}(S)=\lambda_k^{-1}\int_S\left(\sum_i\abs{\nabla u_i}^2\right)\mathit{dVol}_g.
\end{equation}
Suppose the contrary to the statement, that is the singular set $\Sigma$ of $\mu$ is empty. Then, the mass $\mu_{\mathit{abs}}(S)$ equals one. By the hypotheses the complement $M\backslash S$ is a non-empty open set, and comparing relations~\eqref{id:1} and~\eqref{id:2}, we conclude that $\nabla u_i$ vanishes on $M\backslash S$ for any $i=1,\ldots,\ell$. It is then straightforward to see that the $u_i$'s are constant functions on $M\backslash S$, and the sum $\sum u_i^2$ equals one almost everywhere on $M$.

Now the repetition of the argument in the proof of part~(i) shows that the weakly harmonic map 
$$
U:M\ni x\longmapsto (u_1(x),\ldots, u_\ell(x))\in
S^{\ell-1}\subset\mathbf R^\ell
$$
is defined on the whole surface, and by Helein's regularity~\cite{He}, is $C^\infty$-smooth everywhere. Since it is constant on a non-empty open subset $M\backslash S$, by the unique continuation~\cite{Sam}, we conclude that it is constant everywhere. Thus, the sum $\sum\abs{\nabla u_i}^2$ vanishes identically, and by~\eqref{id:2} we obtain a contradiction with the assumption that $\mu_\mathit{abs}$ is a probability measure.
\end{proof}

Finally, mention that Lemma~\ref{nadir} together with the arguments in the proof of Theorem~$C_k$ show that $\lambda_k$-extremal metrics with conical singularities correspond to harmonic maps into a Euclidean sphere defined by a collection of $\lambda_k$-eigenfunctions. This statement is a generalization of the results in~\cite{El03}, see also~\cite{El00,Na96}, known for Riemannian metrics. Due to its importance we state it  below as a corollary.
\begin{corollary}
\label{hm}
Let $(M,c)$ be a compact Riemannian surface, possibly with boundary, and $h$ be a metric with conical singularities conformal to $g\in c$. Then the metric $h$ is $\lambda_k$-extremal if and only if there exists a finite collection of $\lambda_k$-eigenfunctions $(u_i)$, where $i=1,\ldots,\ell$, such that $\sum_i u_i^2=1$, and hence, the map
$$
M\ni x\longmapsto (u_1(x),\ldots, u_\ell(x))\in S^{\ell-1}\subset\mathbf R^\ell
$$
is a harmonic map into a unit sphere in the Euclidean space.
\end{corollary}

\section{Existence of partially regular maximisers}
\label{prm}
\subsection{The main theorem}
Recall that Theorem~$B_1$ states that any $\lambda_1$-maximising sequence of continuous Radon measures converges to a maximal continuous Radon measure $\mu$ provided
\begin{equation}
\label{h:c}
\sup\{\lambda_1(\mu,c)\mu(M):\mu\text{ is a continuous Radon measure  on }M\}>8\pi.
\end{equation}
Due to Theorem~$C_k$ the complete regularity of any maximiser requires the compactness of the embedding
\begin{equation}
\label{h:em}
\mathcal{L}=L_2(M,\mu)\cap L_2^1(M,\mathit{Vol}_g)\subset L_2(M,\mu),
\end{equation}
which, as the results in Sect.~\ref{nonvan} show, is a rather independent hypothesis. As was mentioned earlier,
Nadirashvili and Sire~\cite{Na10a}, and very recently Petrides~\cite{Pet}, announced  the result stating the existence of a completely regular $\lambda_1$-maximiser. Both papers develop a delicate analysis related to the construction of a special maximising sequence that converges to such a maximiser. The purpose of this section is to give a simple argument that proves the existence of a partially regular maximiser.

For a given increasing sequence $(C_n)$ of real numbers such that $C_n\to +\infty$ as $n\to +\infty$, we consider the sets $\mathcal C_n$ formed by continuous Radon measures $\mu$ such that
$$
\mu(B(x,r))\leqslant C_n\cdot r^2
$$
for any closed metric ball $B(x,r)$. Equivalently, the $\mathcal C_n$'s can be described as sets of absolutely continuous measures whose densities $\chi_n$ are bounded above by $C_n$. Clearly, each $\mathcal C_n$ is closed in the weak topology, and thus, contains a measure $\mu_n$ that maximises $\lambda_1(\mu,c)$ in $\mathcal C_n$. If a given conformal class $c$ satisfies the hypothesis~\eqref{h:c}, then, by Theorem~$B_1$ the 
sequence $(\mu_n)$ contains a subsequence that converges weakly to a continuous $\lambda_1$-maximal measure. Moreover, by the results in Sect.~\ref{nonvan}, the measure $\mu$ satisfies the linear isocapacitory inequality and, in paricular, vanishes on sets of zero capacity. Our following result describes further regularity properties of this limit measure.
\begin{TD1}
Let $(M,c)$ be a compact surface, possibly with boundary, endowed with a conformal class $c$ of Riemannian metrics that satisfies the hypothesis~\eqref{h:c}. Let $\mu$ be a continuous $\lambda_1$-maximal measure constructed in the fashion described above, and $S$ be the interior of its support. Then the singular part of $\mu$ is supported in a nowhere dense set $\Sigma$ (of zero Lebesgue measure), and one of the following two possibilities holds:
\begin{itemize}
\item[(i)] either the absolutely continuous part of $\mu$ is trivial, or
\item[(ii)] the absolutely continuous part of $\mu$ has a $C^\infty$-smooth density in $S\backslash\bar\Sigma$ that vanishes at most at a finite number of points on any compact subset in $S\backslash\bar\Sigma$.
\end{itemize}
\end{TD1}

The theorem says that if the maximal measure $\mu$ is not completely singular, than it is the volume measure of a smooth Riemannian metric in $S$, conformal to the ones in $c$, outside of a nowhere dense set of zero Lebesgue measure . As in Theorem~$C_k$, the zeroes of its density  in $S\backslash\bar\Sigma$ correspond to conical singularities of this metric. We decompose the singular set $\Sigma$  into the union of two sets $\Sigma_\mathit{int}$ and $\Sigma_\mathit{out}$, defined as
$$
\Sigma_\mathit{int}=\Sigma\cap S,\quad\text{and}\quad\Sigma_\mathit{out}=\Sigma\backslash S.
$$
Recall that by Theorem~$C_k$, if the embedding~\eqref{h:em} is compact, then $\Sigma_\mathit{int}=\varnothing$. 
In addition, if the complement $M\backslash S$ is non-empty, then $\Sigma_\mathit{out}\ne\varnothing$. These statements indicate on relationships between the isocapacitory inequalities and the properties of the singular set. More precisely, let $\beta(B(x,r),\mu)$ be an isocapacity constant of a closed ball, see Sect.~\ref{nonvan}, and $\Sigma_*$ be the complement in $S$ of a maximal set where $\beta(B(x,r))\to 0$ as $r\to 0$ uniformly in $x$. Then $\Sigma_*$ is a subset of the singular set $\Sigma_\mathit{int}$, and is empty if and only if so is $\Sigma_\mathit{int}$. The last statement here is a consequence of Corollary~\ref{cor:cap}. Alternatively, for a given $\alpha>1$ one can also consider the set $\Sigma_\alpha$ that is the complement in $S$ of a maximal set where
$$
\mu(B(x,r))\ln^\alpha(1/r)\to 0\quad\text{as~} r\to 0
$$
uniformly in $x\in S$. Then, $\Sigma_\alpha\subset\Sigma_\mathit{int}$ and from Lemma~\ref{sc} we conclude that $\Sigma_\alpha$ is empty if and only if so is the singular set $\Sigma_\mathit{int}$.

\subsection{Preliminary considerations}
Let $\mu_n\in\mathcal C_n$ be a probability measure that maximises the first eigenvalue $\lambda_1(\mu,c)$ among all measures in $\mathcal C_n$. By $\chi_n$ we denote its density, and by $\Sigma_n$ the set $\chi_n^{-1}(C_n)$. Changing $\chi_n$ on a zero Lebesgue measure set, we can always assume that the set $\Sigma_n$ is {\em regular} in the following sense: for any $\varepsilon>0$ there exist a closed and open sets $F$ and $G$ such that
\begin{equation}
\label{h:reg}
F\subset\Sigma_n\subset G\qquad\text{and}\qquad \mathit{Vol_g}(G\backslash F)<\varepsilon.
\end{equation}
Let $\mu$ be the weak limit of the measures $\mu_n$, and $S$ be the interior of its support. We fix an open set $D\Subset S$; without loss of generality, we can suppose that it belongs to the support of each $\mu_n$.

Now consider the family of conformal deformations
$$
\mu_{n,t}(X)=\left(\int_Xe^{\phi t}d\mu_n\right)/\left(\int_Me^{\phi t}d\mu_n\right)
$$
with a zero mean-value function $\phi\in L^\infty(M)$ that vanishes on $\Sigma_n$. Since the measures $\mu_{n,t}$ belong to $\mathcal C_n$, we conclude that
\begin{equation}
\label{h:max}
\lambda_1(\mu_{n,t},c)\leqslant\lambda_1(\mu_n,c).
\end{equation}
Clearly, the embedding~\eqref{h:em} is compact for any measure in $\mathcal C_n$, and thus the spaces of first eigenfunctions are non-empty and finite-dimensional. The following claim is essentially a consequence of the first variation formulas (Lemma~\ref{d:eiv}).
\begin{claim}
\label{h:c1}
For each measure $\mu_n$ there exists a finite collection of eigenfunctions $(u_{i,n})$ such that
$$
\sum\limits_i u^2_{i,n}(x)\equiv 1\qquad\text{for any~~} x\in D\backslash\Sigma_n.
$$
\end{claim}
\begin{proof}
Combining Lemma~\ref{d:eiv} with relation~\eqref{h:max}, we conclude that the quadratic form
$$
u\longmapsto\int_{M\backslash\Sigma_n}u^2\phi d\mu_n
$$
is indefinite on the first eigenspace $E$ for any zero mean-value function $\phi\in L^\infty(M\backslash\Sigma_n)$. Now the conclusion follows from a separation argument similar to the one used in the proof of Lemma~\ref{nadir}.
\end{proof}
The following claim yields a formula for the densities $\chi_n$; its proof is a repetition of the argument in the proof of Theorem~$C_k$, see Sect.~\ref{em1}.
\begin{claim}
\label{h:c2}
Under the conditions of Claim~\ref{h:c1}, the eigenfunctions $(u_{i,n})$ are smooth in the interior of $D\backslash\Sigma_n$, and so are the densities $\chi_n$. Moreover, we have the following relation
$$
\chi_n(x)=\left(\sum\limits_i\abs{\nabla u_{i,n}}^2(x)\right)/\lambda_1(\mu_n,c)
$$
for any interior point $x\in D\backslash\Sigma_n$.
\end{claim}
Finally, we need the following statement.
\begin{claim}
\label{h:c3}
The multiplicities of the first eigenvalues $\lambda_1(\mu_n,c)$ are bounded by a quantity that depends on the topology of $M$ only. 
\end{claim}
When the measure $\mu$ is the genuine volume measure of a $C^\infty$-smooth Riemannian metric, the statement is classical and is due to Cheng~\cite{Cheng}. Claim~\ref{h:c3} is a partial case of a more general result proved in~\cite[Sect.~5]{Ko3}.

\subsection{Proof of Theorem~$D_1$}
Denote by $\Sigma_n^*$ the union $\cup_{k\geqslant n }\Sigma_k$. Since the volumes of the $\Sigma_n$'s converge to zero,
$$
\mathit{Vol}_g(\Sigma_n)\leqslant 1/C_n\rightarrow 0 \quad\text{as~ }n\rightarrow +\infty,
$$
then selecting their subsequence, if necessary, we can suppose that so do the volumes of the $\Sigma_n^*$'s. Further, the sequence $\Sigma_n^*$ is nested, and by $\Sigma$ we denote its limit, that is $\cap_n\Sigma_n^*$. Clearly, the limit set $\Sigma$ has a zero Lebesgue measure. Besides, it satisfies property~\eqref{h:reg} and, in particular, is nowhere dense in $M$.

Now let $G$ be an open neighbourhood of $\Sigma$; it also contains sets $\Sigma_n^*$ for a sufficiently large $n$. By Claim~\ref{h:c1}, for any measure $\mu_n$ there exists a collection of eigenfunctions $(u_{i,n})$ such that $\sum_i u^2_{i,n}=1$ on $D\backslash G$, where $D\Subset S$ is a fixed open set. By Claim~\ref{h:c3}, the multiplicities of the eigenvalues $\lambda_1(\mu_n,c)$ are bounded and, choosing a subsequence of the $\mu_n$'s, we can suppose that for each $n\in\mathbf N$ there is the same number of eigenfunctions $(u_{i,n})$, where $i=1,\ldots,m$, such that $\sum_i u^2_{i,n}=1$. In other words, for any measure $\mu_n$, we have a harmonic map
$$
U_n:D\backslash\bar G\ni x\longmapsto(u_{i,n}(x))\in S^{m-1}\subset\mathbf R^m. 
$$
By Claim~\ref{h:c2}, we conclude that their energies are also bounded,
$$
E(U_n):=\int\limits_{D\backslash G}\abs{\nabla U_n}^2d\mathit{Vol}_g\leqslant\lambda_1(\mu_n,c).
$$
Now the {\em bubble convergence theorem}~\cite{SaU,Jo} for harmonic maps applies on any compact subset $F$ in the interior of $D\backslash\bar G$. More precisely, there exists a subsequence, also denoted by $(U_n)$, that converges weakly in $W^{1,2}(F, S^{m-1})$ to a smooth harmonic map $U:F\to S^{m-1}$. Moreover, there exists a finite number of `bubble points' $\{x_1,\ldots,x_\ell\}\subset F$ such that the $U_n$'s converge in $C^\infty$-topology on compact sets in $F\backslash\{x_1,\ldots,x_\ell\}$, and the energy densities $\abs{\nabla U_n}^2$ converge weakly in the sense of measures to $\abs{\nabla U}^2$ plus a finite sum of Dirac measures:
$$
\abs{\nabla U_n}^2\mathit{dVol}_g\rightharpoonup\abs{\nabla U}^2\mathit{dVol}_g+\sum_jm_j\delta_{x_j}.
$$
By the uniqueness of the weak limit, we conclude that the restriction of the limit maximal measure $\mu$ on the interior of $D\backslash\bar G$ has the form
$$
\left(\abs{\nabla U}^2\mathit{dVol}_g+\sum_jm_j\delta_{x_j}\right)/\lambda_1(\mu,c).
$$
However, by Theorem~$B_1$, the maximal measure $\mu$ is continuous and, thus, no `bubble points' can occur in the expression above. Taking smaller sets $G$, we conclude that the limit harmonic map $U$ is well-defined on $D\backslash\bar\Sigma$, and is a finite energy map on the whole $D$. Exhausting the set $S$ (the interior of the support of $\mu$) by sets $D\Subset S$, we further conclude that $U$ extends to it as a finite energy harmonic map. Thus, the maximal measure $\mu$ on $S$ has the form
$$
d\mu=\left(\abs{\nabla U}^2/\lambda_1(\mu,c)\right)d\mathit{Vol}_g+d\mu\lfloor\Sigma_\mathit{int},
$$
where the last term stands for the interior singular part of $\mu$. Finally, if $\abs{\nabla U}\not\equiv 0$, then the zeroes of $\abs{\nabla U}$ correspond to the branch points of $U$; as is known~\cite{Jo,Sa}, there can be only finite number of them on any compact subset in $S\backslash\bar\Sigma$.
\qed

\section{Other related results and remarks}
\label{other}
\subsection{Concentration-compactness of extremal metrics}
The ideas developed in Sect.~\ref{weakmax:ex}-\ref{prm} allow also to analyse the limits of sequences formed by extremal conformal metrics. The following statement is a general result in this direction.
\begin{TEk}
Let $M$ be a closed surface endowed with a conformal class $c$, and $(g_n)$ be a sequence of $\lambda_k$-extremal  smooth metrics in $c$ (possibly with conical singularities) normalised to have a unit volume. Then there exists a subsequence $(g_{n_\ell})$ such that one of the following holds:
\begin{itemize}
\item[(i)] the volume measures $\mathit{Vol}(g_{n_\ell})$ converge weakly to a pure discrete measure supported at $k$ points at most, and
$$
\lim\sup\lambda_k(g_{n_\ell})\leqslant C_*k,
$$
where $C_*$ is the Korevaar constant;
\item[(ii)]
the subsequence $(g_{n_\ell})$ converges smoothly to a Riemannian metric (which may have conical singularities only) away from $k$ points at most where the volumes concentrate.
\end{itemize}
\end{TEk}
\noindent
The proof is based on the characterisation of extremal metrics as harmonic maps into Euclidean spheres (Corollary~\ref{hm}) together with Cheng's multiplicity bounds in~\cite{Cheng}. The argument is similar to the one in the proof of Theorem~$D_1$ and uses the bubble convergence theorem for harmonic maps. The estimate in the case~$(i)$ is a consequence of the remark at the end of Sect.~\ref{weakmax:ex}.

For the case of the first eigenvalue the above result can be significantly sharpened.
\begin{TE1}
Let $M$ be a closed surface endowed with a conformal class $c$, and $(g_n)$ be a sequence of $\lambda_1$-extremal  smooth metrics in $c$ (possibly with conical singularities) normalised to have a unit volume. Then there exists a subsequence $(g_{n_\ell})$ such that one of the following holds:
\begin{itemize}
\item[(i)] the volume measures $\mathit{Vol}(g_{n_\ell})$ converge weakly to a pure Dirac measure $\delta_x$ for some $x\in M$, and $\lambda_1(g_{n_\ell})\to 8\pi$ as $\ell\to +\infty$;
\item[(ii)]
the subsequence $(g_{n_\ell})$ converges smoothly to a $\lambda_1$-extremal metric $g$ (possibly with a finite number of conical singularities) and $\lambda_1(g_{n_\ell})\to\lambda_1(g)$ as $\ell\to +\infty$.
\end{itemize}
\end{TE1}
In particular, the theorem says that the set of conformal $\lambda_1$-extremal metrics whose first eigenvalues are bounded away from $8\pi$ is always compact. The critical value $8\pi$ is the maximal first eigenvalue of unit volume metrics on the $2$-sphere, and as is known (due to the non-compactness of the conformal group $\PSL(2,\mathbf C)$) the maximal metrics on it form a non-compact space. 
This compactness statement can be also viewed as a version of the following result by Montiel and Ros~\cite{MR}: {\em on a compact surface of positive genus each conformal class has at most one metric which admits a minimal immersion into a unit sphere by first eigenfunctions.} Indeed, our statement says that the set of conformal metrics that admit harmonic maps (of energy bounded away from $8\pi$) into a unit sphere by first eigenfunctions is compact. Here we, of course, assume that these metrics are allowed to have conical singularities.

The proof of Theorem~$E_1$ follows closely the line of the argument in~\cite{Ko} where analogous results for Schrodinger eigenvalues have been proved. In fact, the formalism developed in the present work allows to shorten the original proof in~\cite{Ko} significantly. The statement of Theorem~$E_1$ continues to hold when extremal metrics $g_n$ belong to variable conformal classes $c_n$ that lie in a bounded domain of the moduli space of conformal structures on $M$. We refer to~\cite{Ko} for details.

\subsection{Remarks and open questions}

\noindent
{\em 1.} As was mentioned, Nadirashvili and Sire~\cite{Na10a} and Petrides~\cite{Pet} announced the existence of a completely regular $\lambda_1$-maximiser in every conformal class on a closed surface. However, it is important to understand up to what extent {\em any $\lambda_1$-maximal measure is regular}. Recall that, as we saw in Sect.~\ref{em1}, there are $\lambda_1$-extremal completely singular measures. It is extremely interesting to understand whether there are completely singular $\lambda_1$-maximal measures. It seems plausible that such maximal measures do not exist, and moreover, the support of any  $\lambda_1$-maximal measure has to coincide with $M$. Similar questions one can also pose for $\lambda_k$-maximisers. 

\medskip
\noindent
{\em 2.} The properties of the singular set $\Sigma$ of a partially regular maximiser, constructed in Sect.~\ref{prm}, seem to be closely related to the properties of its subsets $\Sigma_*$, where the isocapacity constant $\beta(B(x,r))$ fails to converge to zero uniformly in $x$ as $r\to 0$. It is interesting to know more about the relationship between these sets; in particular, whether it is possible to describe the difference $\Sigma\backslash\Sigma_*$ and the hypotheses when it is empty.  Similarly, the properties of the difference $\Sigma\backslash\Sigma_\alpha$, see Sect.~\ref{prm}, are also very interesting. They could lead to the estimates for the Hausdorff dimension of the singular set $\Sigma$. 

\medskip
\noindent
{\em 3. Maximising eigenvalues among circle-invariant conformal metrics.} One of the possibilities to achieve complete regularity of extremal metrics is to impose extra geometric hypotheses on them. For example, one can consider metrics with symmetries. In the note~\cite{Ko2}, we show how  this works for a class of conformal metrics invariant under a free circle action on the torus. In this setting one can show that for any $k>0$ there exists a circle-invariant metric (in any conformal class $\tilde c$ formed by such metrics), understood as a capacitory Radon measure, which maximises the $k$th eigenvalue among all such measures. Besides, any such $\lambda_k$-extremal metric is
\begin{itemize}
\item[(i)] either completely singular and is supported in a zero Lebesgue measure set which is a union of circle orbits, or
\item[(ii)] it is a genuine metric in $\tilde c$, which is $C^\infty$-smooth in the interior of its support.
\end{itemize}
Mention that here there is no hypothesis on the maximal $\lambda_k$-value, unlike in Theorem~$D_1$. The reason is that any circle-invariant Radon measure has a trivial discrete part. The circle-invariance also implies that the maximal metric (in the case~(ii)) has no conical singularities and, thus, is a genuine Riemannian metric.

More generally, it is interesting to understand how any (possibly partial) symmetry of a $\lambda_k$-extremal metric (in the sense of Sect.~\ref{em1}) improves its regularity properties; cf. the example after Theorem~$C_k$.

\medskip
\noindent
{\em 4. Maximising eigenvalues among conformal metrics with integral curvature bound.} Another example when eigenvalue maximisers have good regularity properties is the extremal problem for conformal metrics with the integral Gaussian curvature bound
\begin{equation}
\label{igc}
\int\abs{K_g}^p\mathit{dVol}_g\leqslant C<+\infty,\quad\text{where~}p>1.
\end{equation}
As is known, see~\cite{Tro} and Appendix in~\cite{CY}, sequences of such conformal metrics of bounded volume satisfy concentration-compactness properties, and the concentration phenomenon can be controlled by positive lower bounds on eigenvalues. For example, there always exists a $C^{0,\alpha}$-smooth $\lambda_1$-maximiser among conformal metrics satisfying~\eqref{igc}. On the other hand, maximising sequences for higher eigenvalues have limits that are $C^{0,\alpha}$-smooth metrics away from a finite number of points. The latter are characterised by the volume concentration and, after an appropriate rescaling, correspond to the metrics on  a collection of "bubble spheres" glued by thin tubes.

\appendix
\section{Appendix: details on Theorems~$\mathbf{A_1}$ and~$\mathbf{A_k}$}
\label{ap:gy}
\subsection{Proof of Theorem~$A_1$}
First, we explain the following version of the result by Yang and Yau~\cite[p.~58]{YY80}. Recall that a measure $\mu$  is called {\em continuous} if the mass of any point $\mu(x)$ is equal to zero.
\begin{prop}
\label{YY}
Let $M$ be a closed Riemann surface and $c$ be the conformal class induced by the complex structure. Suppose that $M$ admits a holomorphic map $\varphi:M\to S^2$ of degree $d$. Then for any continuous  Radon measure $\mu$ on $M$ the first eigenvalue satisfies the estimate 
$$
\lambda_1(\mu,c)\mu(M)\leqslant 8\pi d.
$$
\end{prop}
The key ingredient of the proof is the following lemma,
see~\cite{H70,LY82}.
\begin{HL}
Let $x^i$, $i=1,2,3$, be coordinate functions in $\mathbf R^3$, and $\varphi:M\to S^2\subset\mathbf R^3$ be a conformal map to the unit sphere centred at the origin. Then for any continuous Radon measure $\mu$ on $M$ there exists a conformal diffeomorphism $s:S^2\to S^2$ such that
$$
\int_M(x^i\circ s\circ\varphi)d\mu=0\qquad\text{for any }i=1,2,3. 
$$
\end{HL}
\begin{proof}[Proof of Prop.~\ref{YY}]
Let $s$ be the conformal transformation from the Hersch lemma. Using $(x^i\circ s\circ\varphi)$'s as test functions for the Rayleigh quotient, we obtain
$$
\lambda_1(\mu,c)\int_M(x^i\circ s\circ\varphi)^2d\mu\leqslant\int_M\abs{\nabla(x^i\circ s\circ\varphi)}^2\mathit{dVol}_{g_*}.
$$
Summing up these inequalities over all $\imath$'s and using the identity $\sum(x^i)^2=1$ on the unit sphere, we see that
$$
\lambda_1(\mu,c)\mu(M)\leqslant\sum_i\int_M\abs{\nabla(x^i\circ s\circ\varphi)}^2\mathit{dVol}_{g_*}.
$$
The right-hand side here is the energy of the map $(s\circ\varphi)$, which equals $8\pi d$; see~\cite{EeLe}. 
\end{proof}

Now Theorem~$A_1$ follows by application of the Riemann-Roch theorem in the same fashion as in Yang-Yau~\cite{YY80}. As a consequence, we also obtain a version of Hersch's isoperimetric inequality for continuous
Radon measures on the sphere $S^2$. The estimates of Li and Yau~\cite{LY82} for the first eigenvalue via the conformal volume carry over our setting as well.

\subsection{Proof of Theorem~$A_k$}
Recall that the {\em capacitor} in $M$ is a pair $(F,G)$ of Borel subsets $F\subset G$. Given a reference metric $g\in c$, the {\em capacity} of a capacitor $(F,G)$ is defined as
$$
\CAP(F,G)=\inf\left\{\int_M\abs{\nabla\varphi}^2\mathit{dVol}_g\right\},
$$
where the infimum is taken over all $C^\infty$-smooth functions on $M$ whose support lies in the interior of $G$ and such that $\varphi\equiv 1$ in a neighbourhood of $F$.

The idea of the proof is to find a collection of $(k+1)$ disjoint capacitors $(F_i,G_i)$, that is with the disjoint $G_i$'s, such that
\begin{itemize}
\item[$(i)$] $\mu(F_i)\geqslant v$
\item[$(ii)$] $\CAP(F_i,G_i)\leqslant\kappa$
\end{itemize}
for any $i=0,\ldots, k$ and some positive constants $v$ and $\kappa$. Given such capacitors one directly obtains the bound
\begin{equation}
\label{gy1}
\lambda_k(\mu,c)\leqslant \kappa/v.
\end{equation}
Indeed, any test-function $\varphi_i$ for the capacitor $(F_i,G_i)$ whose Dirichlet integral is not greater than $(\kappa+\varepsilon)$ satisfies the inequality
$$
\int_M\abs{\nabla\varphi_i}^2\mathit{dVol}_g\leqslant (\kappa+\varepsilon)/v\cdot\int_M\varphi_i^2d\mu.
$$
Since the capacitors are disjoint, this inequality holds for any function from the span of the $\varphi$'s, $i=0,\ldots, k$. Thus, we conclude that the $k$th eigenvalue $\lambda_k(\mu,c)$ is not greater than $(\kappa+\varepsilon)/v$, and since $\varepsilon$ is arbitrary, we get the bound~\eqref{gy1}.

The existence of a collection of disjoint capacitors satisfying the hypothesis~$(i)$ for any non-atomic measure is the main result in~\cite{GY99,GNY}. On the other hand, since the capacity is defined with respect to a fixed Riemannian metric, the second hypothesis~$(ii)$ can be often easily demonstrated. Before explaining these ingredients in more detail, we first introduce more notation. 

We regard the surface $M$ as a metric space whose distance $d$ is induced by the path lengths in the metric $g$. By an annulus $A$ in $M$ we call a subset of the following form
$$
\{x\in M: r\leqslant d(x,a)<R\},
$$
where $a\in M$ and $0\leqslant r<R<\infty$. We also use the notation $2A$ for the annulus
$$
\{x\in M: r/2\leqslant d(x,a)<2R\}.
$$
It is a consequence of standard results (see the proof of Theorem~5.3 in~\cite{GNY}) that there exists a constant $Q$ (depending on a reference metric $g$) such that for any open metric ball $B$ the capacity $\CAP(B,2B)$ is not greater than $Q$. It is then straightforward to show that for any annulus $A$ in $M$ one has $\CAP(A,2A)\leqslant 4Q$, see~\cite[Lemma~2.3]{GNY}.

Building on the ideas of Korevaar~\cite{Korv}, Grigor'yan and Yau showed that for any continuous measure $\mu$ one can always find a collection of disjoint annuli $\{2A_i\}$ such that the values $\mu(A_i)$ are bounded below by some positive constant. More precisely, in~\cite{GY99,GNY} they prove the following statement.
\begin{GY}
Let $(M,d)$ be a metric space satisfying the following covering property: there exists a constant $N$ such that any metric ball of radius $r$ in $M$ can be covered by at most $N$ balls of radii $r/2$. Suppose that all metric balls in $M$ are precompact. Then for any continuous Radon measure on $M$ and any positive integer $k$ there exists a collection $\{2A_i\}$, where $i=0,\ldots,k$, of disjoint annuli such that
\begin{equation}
\label{gy2}
\mu(A_i)\geqslant c\mu(M)/k\qquad\text{for any }i,
\end{equation}
where the constant $c$ depends only on $N$.
\end{GY}
Clearly, the metric space $(M,d)$ under consideration satisfies the hypothesis of this theorem, and using~\eqref{gy2} we obtain the bounds
$$
\lambda_k(\mu,c)\mu(M)\leqslant Ck,
$$
where the constant $C$ equals $4Q/c$. Now we show that when $M$ is an orientable surface, the constant $C$ can be chosen in the form $C_*(\gamma+1)$, where $C_*$ is a universal constant and $\gamma$ is the genus of $M$.

Regarding $M$ as a Riemann surface and using the Riemann-Roch theorem, we can find a holomorphic  branch cover $u:M\to S^2$ whose degree is not greater than $(\gamma+1)$. Applying Grigor'yan-Yau theorem to the push-forward measure $\mu^*$ on $S^2$ we find a collection of disjoint annuli $\{2A^*_i\}$ such that
$$
\mu^*(A_i^*)\geqslant c_*\mu^*(S^2)/k.
$$
Besides, we also have
$$
\CAP(A_i^*,2A_i^*)\leqslant 4Q_*
$$
for some constant $Q_*$, where the capacity is understood in the sense of the standard metric on $S^2$. Setting
$$
F_i=u^{-1}(A_i^*)\quad\text{and}\quad G_i=u^{-1}(2A_i^*),
$$
we obtain a collection of disjoint capacitors on $M$ that satisfy~$(i)$ with $v$ equal to $c_*\mu(M)/k$.
Further, since the Dirichlet integral is locally preserved by $u$, we conclude that these capacitors also satisfy~$(ii)$ with $\kappa$ equal to $4Q_*(\gamma+1)$. Now the arguments described above yield the eigenvalue bounds
$$
\lambda_k(\mu,c)\mu(M)\leqslant C_*(\gamma+1)k,
$$
where $C_*$ equals $4Q_*/c_*$. In particular, we see that $\lambda_k(\mu,c)\mu(M)$ is bounded over all conformal classes $c$ and continuous Radon measures $\mu$ on $M$.

\section{Appendix: proofs of statements in Sect.~\ref{em1}}
\label{ap:proofs}
\subsection{Proof of Lemma~\ref{c:eif}}
Recall that, since the integral distances $d(\mu,\mu_n)$ are finite, the $L_2$-spaces, regarded as topological vector spaces, corresponding to the measures $\mu$ and $\mu_n$ coincide. Below by $(\cdot,\cdot)$ and $(\cdot,\cdot)_n$ we denote the scalar products on this space corresponding to $L_2(M,\mu)$ and $L_2(M,\mu_n)$ respectively. We claim that the Dirichlet form
$$
D[u]=\int_M\abs{\nabla u}^2\mathit{dVol}_g
$$
is closed with respect to each of the scalar products above. Indeed, by Prop.~\ref{ex:eigen}, the first eigenvalue $\lambda_1(\mu,c)$ does not vanish, and for any $u$ with zero mean-value we have
$$
\int_Mu^2d\mu\leqslant\lambda_1^{-1}(\mu,c)\cdot\int_M\abs{\nabla u}^2d\mathit{Vol}_g.
$$
Now the closeness on the zero mean-value $u$'s follows from the completeness of the space $L^1_2(M,\mathit{Vol}_g)$ modulo constants,  see~\cite{Ma}. Since $D[u]$ vanishes on constants, it is also closed on the whole $L_2$-space. The same argument also yields the claim for the measures $\mu_n$.

Now we apply the representation theorem in~\cite[Chap.~VI]{Kato} to the closed symmetric form $D[u]$ to conclude that there exist closed self-adjoint operators $T$ and $T_n$ such that
$$
D(u,v)=(Tu,v),\qquad D(u,v)=(T_nu,v)_n.
$$
It is straightforward to see that the eigenvalues of $T$ and $T_n$ coincide with $\lambda_k(\mu,c)$ and $\lambda_k(\mu_n,c)$ respectively, and so do their eigenspaces. Further, since the topologies induced by the scalar products $(\cdot,\cdot)$ and $(\cdot,\cdot)_n$ coincide, the operators $T_n$ are also closed in $L_2(M,\mu)$. From the definition of the integral distance we obtain
$$
\abs{1-(T_nu,u)/(Tu,u)}\leqslant\delta(\mu,\mu_n)
$$
for any non-constant $u\in L_2(M,\mu)$. Now the perturbation theorem~\cite[Chap.~VI, Th.~3.6]{Kato} applies, and we conclude that $T_n\to T$ in a generalised sense as closed operators, and the corresponding spectral projectors converge in the norm topology.
\qed
\begin{remark*}
Mention that, in fact, a stronger statement holds: for any $k$ there exists a constant $C(k)$ such that
\begin{equation}
\label{o_bound}
\abs{\Pi_k-\Pi_{n,k}}\leqslant C(k)\cdot\delta(\mu,\mu_n)
\end{equation}
for any sufficiently large $n$. Indeed, by~\cite[Chap.~VI, Th.~3.4]{Kato} the resolvents of $T$ and $T_n$ at the point $(-1)$ satisfy the relation
$$
\abs{\mathcal R(-1,T)-\mathcal R(-1,T_n)}\leqslant C\cdot\delta(\mu,\mu_n).
$$
Further, by the results in~\cite[Chap.~IV]{Kato} the difference $(\mathcal R(\zeta,T)-\mathcal R(\zeta,T_n))$, where $\zeta$ ranges over a compact subset of the common resolvent set, can be estimated in the same fashion for a sufficiently large $n$. Now relation~\eqref{o_bound} follows from the fact that the eigenspace projections are integrals of the resolvents over a small closed curve bounding a region containing $\lambda_k(\mu,c)$ and $\lambda_k(\mu_n,c)$ .
\end{remark*}

\subsection{Proof of Claim~\ref{c1}}
We demonstrate the proof of the first relation; the second follows by similar arguments. Denote by $\Lambda_t$ the sum of all eigenspaces corresponding to $\lambda_i(\mu_t,c)<\lambda_k(\mu_t,c)$, where $i<k$, and by $P_t$ and $P_t^*$ the orthogonal projections on it in $L_2(M,\mu)$ and $L_2(M,\mu_t)$ respectively. Define the modified Rayleigh quotient $\bar{\matheur R}_c(u,\mu_t)$ as 
$$
\left(\int_M\abs{\nabla(u-P_t^*u)}^2\mathit{dVol}_g\right)/\left(\int_M\abs{u-P_t^*u}^2d\mu_t\right).
$$
Clearly, the following relation holds:
$$
\lambda_k(\mu_t,c)=\inf_{u}\bar{\matheur R}_c(u,\mu_t),
$$
where the infimum is taken over all non-trivial $u$ that do not lie in $\Lambda_t$. The first inequality of the claim is a straightforward consequence of the following relation
$$
\bar{\matheur R}_c(u,\mu_t)=\matheur R_c(u,\mu_t)+o(t)\qquad
\text{as}\quad t\to 0,
$$
where $u\in E_k$, and $o(t)$ denotes the quantity such that $o(t)/t$ converges to zero uniformly in $u\in E_k\backslash\{0\}$. Denote by $\Delta(t)$ the difference of the Rayleigh quotients $\bar{\matheur R}_c(u,\mu_t)-\matheur R_c(u,\mu_t)$; it is given by the formula
$$
\Delta(t)={\matheur R}_c(u,\mu_t)\cdot\left(\int |P_t^*u|^2d\mu_t\right)/\left(\int |u-P_t^*u|^2d\mu_t\right).
$$
Now by the remark after Lemma~\ref{c:eif}, for a proof of the claim it is sufficient to show that
\begin{equation}
\label{o:inek}
\int |P_t^*u|^2d\mu_t\leqslant\left(\int u^2d\mu\right)\cdot O(t^2)\qquad\text{as}\quad t\to 0
\end{equation}
for any $u\in E_k\backslash\{0\}$. To see that this holds, choose a basis $(e_{i,t})$ for the space $\Lambda_t$ orthogonal in $L_2(M,\mu_t)$ and normalised in $L_2(M,\mu)$. By $(e_i)$ we denote the corresponding basis at $t=0$. Then for any $u\in E_k\backslash\{0\}$, we have
$$
\int |P_t^*u|^2d\mu_t\leqslant\max_i\left(\int\abs{e_{i,t}}^2d\mu_t\right)\cdot\sum\limits_i\left(\int e_{i,t}ud\mu_t-\int e_iud\mu\right)^2.
$$
By Cauchy's inequality, each term in the sum on the right-hand side can be estimated by twice the sum
$$
\left(\int e_{i,t}ud\mu_t-\int e_{i,t}ud\mu\right)^2+\left(\int(e_{i,t}-e_i)ud\mu\right)^2.
$$
Finally, each term here can be now estimated by the right-hand side in~\eqref{o:inek}: for the first it follows from the definition of $\mu_t$, for the second -- from the inequality
$$
\int(e_{i,t}-e_i)^2d\mu\leqslant 4\abs{P_{i,t}-P_i}^2,
$$
see~\cite[Chap.~IV]{Kato}, and relation~\eqref{o_bound}.
\qed
\begin{remark*}
For the case of the first eigenvalue estimate~\eqref{o:inek} can be proved directly, without appealing to Kato's perturbation theory and relation~\eqref{o_bound}. Indeed, in this case the lower eigenspaces coincide and, hence, the difference $(P_{i,t}-P_i)$ is identically zero.
\end{remark*}

\subsection{Proof of Claim~\ref{c2}}
Let $\Pi_t$ be the orthogonal projection onto $E_t$ in $L_2(M,\mu)$. By Lemma~\ref{c:eif}, for a proof of the claim it is sufficient to show that the family $L_\phi(\Pi_tu,\mu)$ converges to the quantity $L_\phi(u,\mu)$ as $t\to 0$ uniformly in $u\in E_k\backslash\{0\}$. Denote by $\matheur Q(u,\mu)$ the quotient
$$
\left(\int_Mu^2\phi d\mu\right)/\left(\int_Mu^2d\mu\right).
$$
By the triangle inequality, we obtain
\begin{multline}
\label{only}
\abs{L_\phi(u,\mu)-L_\phi(\Pi_tu,\mu)}\leqslant\lambda_k(\mu)\abs{\matheur Q(u,\mu)-\matheur Q(\Pi_tu,\mu)}\\
+\abs{\phi}_\infty\abs{\matheur R_c(u,\mu)-\matheur R_c(\Pi_tu,\mu)},
\end{multline}
where $\abs{~\cdot}_\infty$ stands for the $L_\infty$-norm. By Lemma~\ref{c:eif} we conclude that the quotient
$$
\left(\int_M(\Pi_tu)^2 d\mu\right)/\left(\int_Mu^2d\mu\right)
$$
converges to $1$ uniformly in $u\in E_k\backslash\{0\}$. Using this, it is straightforward to estimate the first term on the right-hand side in~\eqref{only} by  the quantity $\lambda_k(\mu)\abs{\phi}_\infty$ times the sum
$$
\abs{1-\left(\int_Mu^2d\mu\right)/\left(\int_M(\Pi_tu)^2 d\mu\right)}+ C\left(\int_M\abs{u^2-(\Pi_tu)^2}d\mu\right)/\left(\int_Mu^2d\mu\right)
$$
for all sufficiently small $t$. By the discussion above the first term here converges to zero uniformly over non-trivial $u\in E_k$, and by Lemma~\ref{c:eif} so does the second term. Further, the term involving the difference of the Rayleigh quotients on the right hand-side in~\eqref{only} can be estimated in the following fashion:
$$
\abs{\matheur R_c(u,\mu)-\matheur R_c(\Pi_tu,\mu)}\leqslant\abs{\lambda_k(\mu)-\lambda_k(\mu_t)}+\abs{\matheur  R_c(\Pi_tu,\mu_t)-\matheur R_c(\Pi_tu,\mu)},
$$
where the second term is bounded by $\lambda_k(\mu_t)\delta(\mu,\mu_t)$. Thus, we see that it also converges to zero uniformly in $u$.
\qed


{\small
\begin{thebibliography}{99}
\addcontentsline{toc}{section}{References}

\bibitem{AH} Adams,~D.~R., Hedberg,~L. {\em Function spaces and potential theory.} Grundlehren der Mathematischen Wissenschaften, 314. Springer-Verlag, Berlin, 1996. xii+366 pp.

\bibitem{Be73} Berger,~M. {\em Sur les premi\`eres valeurs propres des vari\'et\'es Riemanniennes.} Compositio Math. {\bf 26} (1973), 129--149.

\bibitem{CY} Chang,~S.-Y., Yang,~P. {\em Isospectral conformal metrics on $3$-manifolds.} J. Amer. Math. Soc. {\bf 3} (1990), 117--145.

\bibitem{Chen} Chen, M.-F. {\em Eigenvalues, inequalities, and ergodic theory.} Probability and its Applications (New York). Springer-Verlag London, Ltd., London, 2005, xiv+228 pp.

\bibitem{Cheng} Cheng,~S.~Y. {\em Eigenfunctions and nodal sets.} Comment. Math. Helv. {\bf 51} (1976), 43--55.

\bibitem{CEl} Colbois,~B., El Soufi,~A. {\em Extremal eigenvalues of the Laplacian in a conformal class of metrics: the `conformal spectrum'.}  Ann. Global Anal. Geom.  {\bf 24} (2003), 337--349.

\bibitem{CEG} Colbois,~B., El Soufi,~A., Girouard,~A. {\em Isoperimetric control of the Steklov spectrum.}
J. Funct. Anal. {\bf 261} (2011), 1384--1399.

\bibitem{DM} Dellacherie,~C., Meyer,~P.-A. {\em Probabilties and potential.} North-Holland Mathematics Studies, 29. North-Holland Publishing Co., Amsterdam -- New York, 1978, viii+189 pp.

\bibitem{EeLe} Eells,~J., Lemaire,~L. {\em Selected topics in harmonic maps.} CBMS Regional Conference Series in Mathematics, 50. Published for the Conference Board of the Mathematical Sciences, Washington, DC; by the American Mathematical Society, Providence, RI, 1983. v+85 pp.

\bibitem{El00} El Soufi,~A., Ilias,~S. {\em Riemannian manifolds admitting isometric immersions by their first eigenfunctions.} Pacific J. Math. {\bf 195} (2000), 91--99.

\bibitem{El03} El Soufi,~A., Ilias,~S. {\em Extremal metrics for the first eigenvalue of the Laplacian in a conformal class.} Proc. Amer. Math. Soc. {\bf 131} (2003), 1611-1618.

\bibitem{El06} El Soufi,~A., Giacomini,~H., Jazar,~M. {\em A unique extremal metric for the least eigenvalue of the Laplacian on the Klein bottle.} Duke Math. J. {\bf 135} (2006), 181--202.

\bibitem{FS} Fraser,~A., Schoen,~R. {\em The first Steklov eigenvalue, conformal geometry, and minimal surfaces.} Adv. Math., to appear.

\bibitem{FS2} Fraser,~A., Schoen,~R. {\em Sharp eigenvalue bounds and minimal surfaces in the ball.} arXiv: 1209.3789.

\bibitem{FS3} Fraser,~A., Schoen,~R. {\em Minimal surfaces and eigenvalue problems.} arXiv:1304.0857.

\bibitem{Gi} Girouard,~A. {\em Fundamental tone, concentration of density, and conformal degeneration of surfaces.} Canad. J. Math. {\bf 61} (2009), 548--565.

\bibitem{GP} Girouard,~A., Polterovich, I. {\em Shape optimization for low Neumann and Steklov eigenvalues.}
Math. Methods Appl. Sci. {\bf 33} (2010), 501--516.

\bibitem{GY99} Grigor'yan,~A., Yau,~S.-T. {\em Decomposition of a metric space by capacitors.} ``Differential equations: La Pietra 1996'', Ed. Giaquinta et. al., Proceedings of Symposia in Pure Mathematics, {\bf 65}, 1999, 39--75.

\bibitem{GNY} Grigor'yan,~A., Netrusov,~Y., Yau~S.-T. {\em Eigenvalues of elliptic operators and geometric applications.} Surveys in differential geometry. Vol.~IX, 147--217. Surv. Diff. Geom., IX, Int. Press, Somerville, MA, 2004.

\bibitem{H70} Hersch,~J. {\em Quatre propri\'et\'es isop\'erim\'etriques de membranes sph\'eriques homog\`enes.} C. R. Acad. Sci. Paris S\'er. A-B {\bf 270} (1970), A1645--A1648.

\bibitem{He} H\'elein,~F. {\em Harmonic maps, conservation laws and moving frames.} Translated from 1996 French original. Second edition. Cambridge Tracts in Mathematics, 150. Cambridge University Press, Cambridge, 2002. xxvi+264 pp.

\bibitem{JLNNP} Jakobson,~D., Levitin,~M., Nadirashvili,~N., Nigam,~N., Polterovich,~I. {\em How large can the first eigenvalue be on a surface of genus two?} Int. Math. Res. Not. (2005), 3967--3985.

\bibitem{JNP06} Jakobson,~D., Nadirashvili,~N., Polterovich,~I. {\em Extremal metric for the first eigenvalue on the Klein bottle.} Canad. J. Math. {\bf 58} (2006), 381--400.

\bibitem{Jam} Jammes,~P. {\em Spectre et g\'eom\'etrie conforme des vari\'et\'es compactes \`a bord.} arXiv:1204.5978.

\bibitem{Jo} Jost,~J. {\em Two-dimensional geometric variational problems.} Pure and Applied Mathematics. John Wiley \& Sons, Ltd.,
Chichester, 1991. x+236 pp.

\bibitem{Kato} Kato,~T. {\em Perturbation theory for linear operators.} Second edition. Grundlehren der Mathematischen Wissenschaften, Band~132. Springer-Verlag, Berlin--New York, 1976. xxi+619 pp. 

\bibitem{Ko} Kokarev,~G. {\em On the concentration-compactness phenomenon for the first Schrodinger eigenvalue.} Calc. Var. Partial Differential Equations {\bf 38} (2010), 29--43.

\bibitem{Ko3} Kokarev,~G. {\em On multiplicity bounds for Schrodinger eigenvalues on Riemannian surfaces.}
arXiv:1310.2207.

\bibitem{Ko2} Kokarev,~G. {\em Maximising Laplace eigenvalues among circle-invariant metrics on Riemannian surfaces.} Preprint, 2012.

\bibitem{KoNa} Kokarev,~G., Nadirashvili,~N. {\em On the first Neumann eigenvalue bounds for conformal metrics.} International Mathematical Series, vol. 11-13, Topics around the Research of Vladimir Maz'ya, Springer, 2010, 229--238.

\bibitem {Korv} Korevaar,~N. {\em Upper bounds for eigenvalues of conformal metrics} J. Diff. Geom., {\bf 37} (1993), 73--93.

\bibitem{LY82} Li,~P., Yau,~S.-T. {\em A new conformal invariant and its applications to the Willmore conjecture and the first eigenvalue of compact surfaces.} Invent. Math. {\bf 69} (1982), 269--291. 


\bibitem{Ma} Maz'ja,~V.~G. {\em Sobolev spaces.} Translated from the Russian by T. O. Shaposhnikova. Springer Series in Soviet Mathematics. Springer-Verlag, Berlin, 1985. xix+486 pp. 

\bibitem{MR} Montiel,~S., Ros,~A. {\em Minimal immersions of surfaces by the first eigenfunctions and conformal area.} Invent. Math. {\bf 83} (1985), 153--166.

\bibitem{Na96} Nadirashvili,~N. {\em Berger's isoperimetric problem and minimal immersions of surfaces.} Geom. Funct. Anal. {\bf 6} (1996), 877--897.

\bibitem{Na10a} Nadirashvili,~N., Sire,~Y. {\em Conformal spectrum and harmonic maps.} arXiv:1007.3104.

\bibitem{Pet} Petrides,~R. {\em Existence and regularity of maximal metrics for the first Laplace eigenvalue on surfaces.} arXiv: 1310.4697.

\bibitem{Re} Reshetnyak,~Y.~G. {\em The concept of capacity in the theory of functions with generalised derivatives.} (Russian)
Sib. Math. J. {\bf 10} (1969), 1109--1138.

\bibitem{SaU} Sacks,~J., Uhlenbeck,~K. {\em The existence of minimal immersions of $2$-spheres.} Ann. of Math. (2) {\bf 113} (1981), 1--24. 

\bibitem{Sa} Salamon, S. {\em Harmonic and holomorphic maps.} Geometry seminar "Luigi Bianchi" II--1984, 161--224, Lecture Notes in Math., 1164, Springer, Berlin, 1985.

\bibitem{Sam} Sampson,~J.~H. {\em Some properties and applications of harmonic mappings.} Ann. Sci. \'Ecole Norm. Sup. (4) {\bf 11} (1978), 211--228.


\bibitem{Tro} Troyanov,~M. {\em Un principe de concentration-compacit\'e pour les suites de surfaces riemanniennes. } Ann. Inst. H. Poincar\'e Anal. Non Lin\'eaire {\bf 8} (1991), 419--441. 

\bibitem{Wei} Weinstock,~R. {\em Inequalities for a classical eigenvalue problem.} J. Rational Mech. Anal. {\bf 3} (1954), 745--753.

\bibitem{YY80} Yang,~P., Yau,~S.-T. {\em Eigenvalues of the Laplacian of compact Riemann surfaces and minimal submanifolds.} Ann. Scuola Norm. Sup. Pisa Cl. Sci. (4) {\bf 7} (1980), 55--63.

\end{thebibliography}
}
\end{document}